\documentclass[draft]{article}
\usepackage{amsmath,amsfonts,latexsym, amssymb,   mathrsfs, eucal}

\usepackage{dsfont}
\usepackage{stmaryrd}

\usepackage{color}

\newcommand{\fa}{{\frak a}} 

\newcommand{\ttau}{\vartheta}

\newcommand{\wt}{\widetilde}
\newcommand{\wh}{\widehat}

\newcommand{\beqa}{\begin{eqnarray}}
\newcommand{\eeqa}{\end{eqnarray}}
\newcommand{\e}{\varepsilon}

\newcommand{\pt}{\partial}
\newcommand{\rd}{{\rm d}}
\newcommand{\bR}{{\mathbb R}}

\newcommand{\non}{\nonumber}

\newcommand{\abs}[1]{\left| #1 \right|}

\newcommand{\tr}{\mbox{Tr\,}}

\newcommand{\bx}{{\bf{x}}}
\newcommand{\by}{{\bf{y}}}

\newcommand{\bv}{{\bf{v}}}
\newcommand{\bw}{{\bf{w}}}
\newcommand{\bz}{{\bf {z}}}

\newcommand{\al}{\alpha}

\newcommand{\be}{\begin{equation}}
\newcommand{\ee}{\end{equation}}

\newcommand{\barv}{ {[ v ]}}
\newcommand{\barZ}{ {[ Z ]}}

\newcommand{\la}{\lambda}

\newcommand{\om}{{\omega}}

\renewcommand{\th}{\theta}

\newcommand{\cL}{{\mathscr L}}

\newcommand{\cG}{{\mathcal G}}

\newcommand{\cH}{{\mathcal H}}

\newcommand{\ov}{\overline}

\newcommand{\re}{{\mathfrak{Re} \, }}
\newcommand{\im}{{\mathfrak{Im} \, }}
\newcommand{\E}{{\mathbb E }}
\newcommand{\R}{{\mathbb R }}
\newcommand{\N}{{\mathbb N}}

\newcommand{\CC}{{\mathbb C }}
\newcommand{\RR}{{\mathbb R }}
\newcommand{\NN}{{\mathbb N}}

\newcommand{\bl}{{\boldsymbol \lambda}}

\newcommand{\bt}{{\boldsymbol \theta}}

\renewcommand{\P}{{\mathbb P}}

\newcommand{\ind}{{\,\mathrm{d}}}

\newcommand{\bS}{ {\bf  S}}

\newtheorem{theorem}{Theorem}
\newtheorem{corollary}[theorem]{Corollary}
\newtheorem{lemma}[theorem]{Lemma}

\newcommand{\qed}{\hfill\fbox{}\par\vspace{0.3mm}}
\newenvironment{proof}{{\bf Proof.}} {\hfill\qed}

\numberwithin{equation}{section}
\numberwithin{theorem}{section}
\numberwithin{definition}{section}
\numberwithin{remark}{section}

\newcommand{\var}{{\mathrm{var}}}

\newcommand{\deq}{\mathrel{\mathop:}=}

\renewcommand{\varpi}{\vartheta}

\title{Universality of local spectral statistics of random matrices}

\author{
L\'aszl\'o Erd\H os${}^1$\thanks{Partially supported
by SFB-TR 12 Grant of the German Research Council}, 
Horng-Tzer Yau${}^2$\thanks{Partially supported
by NSF grants DMS-0757425, 0804279} 
 \\\\
Institute of Mathematics, University of Munich, \\
Theresienstr. 39, D-80333 Munich, Germany \\ lerdos@math.lmu.de ${}^1$ \\ \\
Department of Mathematics, Harvard University\\
Cambridge MA 02138, USA \\  htyau@math.harvard.edu ${}^2$
}

\begin{document}
\date{Jan 30, 2012}

\maketitle

\begin{abstract}
The Wigner-Dyson-Gaudin-Mehta   conjecture asserts that the local eigenvalue 
statistics of large random matrices exhibit
universal behavior  depending only on the symmetry
class of the matrix ensemble. For invariant matrix models, the eigenvalue distributions are given 
by a log-gas with  potential $V$ and inverse temperature $\beta = 1, 2, 4$, corresponding to 
the orthogonal, unitary and symplectic ensembles. 
For $\beta \not  \in \{1, 2, 4\}$, there is no natural random matrix  ensemble  behind this model,
 but the statistical physics interpretation of the log-gas
is still valid for all $\beta > 0$. The universality conjecture for invariant ensembles asserts that  
the local eigenvalue statistics are independent of $V$.  
In this article, 
we review our recent solution to the universality conjecture for both  
 invariant and non-invariant ensembles. We will also demonstrate that the  local ergodicity 
of the  Dyson Brownian motion   is   the intrinsic  mechanism behind the 
universality.  Furthermore, we review the solution of Dyson's conjecture on the local relaxation time 
of the   Dyson Brownian motion.  Related questions 
such as delocalization of eigenvectors and local version of Wigner's semicircle law
will also be discussed. 
\end{abstract}

{\bf AMS Subject Classification (2010):} 15B52, 82B44

\medskip

{\it Keywords:}  Random matrix, local semicircle law,
Tracy-Widom distribution, Dyson Brownian motion.

\medskip
\medskip

\begin{minipage}[c]{5.5in}
{\it ``Perhaps I am now too courageous when I try to guess the distribution of the distances between
successive levels (of energies of heavy nuclei).   Theoretically, the 
situation is quite simple if one attacks the problem  in a 
simpleminded fashion.  The question is simply what are the 
 distances  of the characteristic values of  a symmetric
matrix with random coefficients."  }
\end{minipage}

\medskip 
\centerline{\qquad\qquad\qquad\qquad\qquad\qquad Eugene Wigner on 
the Wigner surmise, 1956 }

\section{Introduction}

What do  the  eigenvalues of a typical  large matrix look like? 
Do we expect certain universal patterns of eigenvalue statistics to emerge? Although random matrices 
appeared already in a concrete statistical application by Wishart in 1928 \cite{Wish}, 
these  natural questions were not raised until the pioneering work \cite{W} of E. Wigner in
the 1950's.
To make the problem simpler, 
we restrict ourselves to either real symmetric or complex Hermitian matrices 
so that the eigenvalues are real.  For definiteness, we 
consider  $N\times N$ square  matrices 
$H=H^{(N)} = (h_{ij})$ with matrix elements having mean zero and variance $1/N$, i.e., 
\be
   \E\, h_{ij} = 0,   \quad \E |h_{ij}|^2= \frac{1}{N}
 \qquad i,j =1,2,\ldots, N.
\label{centered}
\ee
The random variables  $h_{ij}$,   $i,j=1, \ldots, N$  are 
real or  complex  independent random variables 
subject to 
the symmetry constraint $h_{ij}= \ov h_{ji}$.   These ensembles of random matrices 
are called {\it Wigner matrices.}  We will always consider
the limit 
as the matrix size goes to infinity, i.e., $N\to\infty$.

The first rigorous result about the spectrum of a random matrix of this type
is the famous {\it Wigner semicircle law}
 \cite{W} which
  states that the  empirical densities of the eigenvalues, $\lambda_1, \lambda_2, \ldots,
\lambda_N$,  of  large 
 symmetric or Hermitian matrices, after proper normalization
such as \eqref{centered}, are given by 
\be\label{sc}
\varrho_N (x) := \frac{1}{N}\sum_{j=1}^N \delta(x-\lambda_j)\rightharpoonup
\varrho_{sc} (x) := \frac 1 { 2 \pi} \sqrt {(4 - x^2)_+}
\ee
in the weak limit as $N \to \infty$. The limit density is
independent  of the details of the distribution of $h_{ij}$.
The motivation for Wigner was to find a  phenomenological
 model for the  energy gap statistics   of large atomic nuclei since 
 the energy levels of large quantum systems 
are impossible to compute from first principles.
After several attempts, Wigner was convinced  that random matrices were the right models. 
 Besides the semicircle law, he also predicted that the 
 eigenvalue gap distribution in the bulk of the spectrum is given by the {\it Wigner surmise}, 
e.g. in the case of symmetric matrices,
$$
   \P\Big( \frac{s}{N\varrho} \le \lambda_j-\lambda_{j-1} 
\le  \frac{s+\rd s}{N\varrho}\Big) \approx 
\frac{\pi s}{2}
\exp\big( -  \frac{\pi}{4} s^2\big)\rd s,
$$
where $\varrho$ is the local density of eigenvalues  (see \cite{M} for an overview).

Wigner's proof of the semicircle law was  a
 moment method via computing $\E\, \tr H^n$ for each $n$.
The Wigner surmise was much more difficult to understand.
In the pioneering work 
by Gaudin \cite{Gau} 
 the exact gap distributions of random matrices with Gaussian 
distribution for matrix elements were 
computed in terms  of a Fredholm determinant 
involving  Hermite polynomials. Hermite polynomials were 
first  introduced in the context of random matrices by Mehta and Gaudin \cite{MG} earlier.
 Dyson and Mehta \cite{M2, Dy1, Dy2}
have later extended this exact calculation 
to correlation functions and to other symmetry classes.
 To keep our presentation simple, 
we state the corresponding results in terms of the
eigenvalue correlation functions for Hermitian 
$N\times N $  matrices.
If  $p_N(\lambda_1, \lambda_2, \ldots , \lambda_N)$ denotes the
joint probability density of the  (unordered) eigenvalues, then
the $n$-point correlation functions (marginals) are defined by
\be
  p^{(n)}_N(\la_1, \la_2, \ldots, \la_n):
 = \int_{\bR^{N-n}} p_N( \la_1, \ldots,\la_n, \la_{n+1},
\ldots \la_N) \rd\la_{n+1} \ldots \rd\la_{N}.
\label{pk}
\ee
In the Gaussian case, the joint probability
density of the eigenvalues can be expressed explicitly as
\be
    p_N(\lambda_1, \lambda_2, \ldots , \lambda_N) 
  = \mbox{const.} \prod_{i<j} (\lambda_i-\lambda_j)^2 \prod_{j=1}^N
   e^{- \frac{1}{2}N  \sum_{j=1}^N \lambda_j^2}.
\label{expli28}
\ee
The Vandermonde determinant structure allows one to compute
the $k$-point  correlation functions in the large $N$ limit via 
Hermite polynomials that are the
orthogonal polynomials with respect to the Gaussian weight function.

The result of Dyson, Gaudin and Mehta  asserts that  for 
any fixed energy $E$ in the bulk of the spectrum, i.e., $|E| < 2$,  
the small scale behavior of $ p^{(n)}_N$ is given explicitly by 
\be
  \frac{1}{[\varrho_{sc}(E)]^n}
 p_N^{(n)}\Big( E+ \frac{\al_1}{N\varrho_{sc} (E)}, E + \frac{\al_2}{N\varrho_{sc}(E)},
 \ldots ,E+ \frac{\al_n}{N\varrho_{sc}(E)}\Big) \rightharpoonup  
\det \big( K(\al_i - \al_j)\big)_{i,j=1}^n
\label{sine}
\ee
where $K$ is the  celebrated sine kernel
\begin{align}
K(x,y)
  =  \frac{\sin  \pi (x-y)}{\pi(x-y)}.
\end{align}
Note that the limit in \eqref{sine} is independent of the energy $E$
as long as it is in the bulk of the spectrum. 
 The rescaling by a factor $N^{-1}$ of the correlation functions in \eqref{sine}
corresponds to the typical distance between consecutive eigenvalues 
and we will refer to  the law under  such scaling
 as  {\it local  statistics}.
Similar but much more complicated  formulas for symmetric matrices were also obtained. 
It is well-known that the eigenvalue gap distribution can be computed from the 
correlation 
functions via the inclusion-exclusion principle
and thus \eqref{sine} also  yields a precise asymptotics for eigenvalue
 gap distributions. 
 In a striking coincidence, 
the Wigner surmise, which was based on a $2\times 2$ matrix computation,
 agrees with this sophisticated 
formula with a  typical error of only a few percentage points. 
 Note that the correlation functions do not factorize, i.e.
the eigenvalues are strongly correlated despite that the matrix elements
are independent. Eigenvalues of random matrices thus represent
a strongly correlated point process obtained from independent random
variables in a natural way.

The central thesis of Wigner  is the belief  that
the eigenvalue gap distributions   for  large complicated quantum systems are universal
 in the sense that they  depend only on the symmetry class
 of the physical system but not  on other  detailed structures. 
 This thesis has never been proved for  any truly interacting system and 
there is  even no  heuristically  convincing argument for its correctness.
Despite this, there is a general belief  that the random matrix statistics and
 Poisson statistics represent
two paradigms  of energy level statistics for many-body quantum systems:
Poisson for independent systems and random matrix for highly correlated systems. 
In fact, these paradigms extend even to certain one-body systems
such as the quantization of the geodesic flow in a domain or on a manifold  
\cite{BT, BGS} or random Schr\"odinger operators \cite{Spe}.

In retrospect, Wigner's idea should have received even more attention.
For centuries, the primary 
territory of probability theory was to model uncorrelated
or weakly correlated systems via the law of large numbers or the 
central limit theorem. Random matrix statistics  
is essentially the first and only  general  
computable pattern  for complicated correlated systems and
 it is conjectured to be ubiquitous.  
We only mention here the spectacular result of Montgomery \cite{Mont}
 which proves a special  case of the conjecture (under the assumption of 
the Riemann hypothesis) 
 that  the distribution 
of zeros of the Riemann zeta function
on the critical line is given by a random matrix statistics.

The simplest class to test  Wigner's universality  hypothesis upon is  the random 
matrix ensemble itself.
 All calculations by Dyson, Gaudin and Mehta  are for Gaussian ensembles, i.e.,
where the matrix elements $h_{ij}$ are  real or complex 
Gaussian random variables.
 These ensembles are called the Gaussian orthogonal ensemble (GOE)
 and Gaussian unitary ensemble (GUE). 
If Wigner's universality hypothesis is correct, then the local eigenvalue 
statistics should be independent of the law of the matrix elements.
This is generally referred to as the {\it universality conjecture 
 of  random matrices}  and we will call it the {\it Wigner-Dyson-Gaudin-Mehta conjecture} 
due to the vision of Wigner and the pioneering work of these authors.
 It was first formulated in Mehta's treatise on
random matrices \cite{M} in 1967 and has remained a  key question
in the subject ever since.
 Our goal in this paper is to review  the recent progress 
in this direction and sketch some of the important ideas.   

 Random matrices have been  intensively studied in the last 15-20 years
and we will not be able to present all aspects of this research. We refer
the reader to recent comprehensive books  \cite{De1, DG1, AGZ}.

\medskip

The laws of random matrices can be generally divided into {\it invariant} 
and {\it non-invariant ensembles}.
The invariant ensembles 
are characterized by a probability  measure of the form  
$Z^{-1} e^{- N \beta  { \rm Tr} V(H)/2 } \rd H$ 
where $N$ is the size of the matrix, 
 $V$ is a real valued potential and $Z$ is the normalization constant.  
The parameter 
$\beta>0$ is  determined by the symmetry class of the model and 
$\rd H$ is the Lebesgue measure on matrices in the class. 
These ensembles are called invariant since the probability law depends only on the trace
 of a  function of the matrix and 
thus is invariant under changes of coordinates. 
The matrix elements 
are in general correlated and they are  independent if only if the model is Gaussian,
 i.e., $V$ is quadratic.

For invariant ensembles, the probability distribution of the eigenvalues
$\lambda=(\lambda_1,\dots,\lambda_N)$ with $\lambda_1\leq \dots\leq \lambda_N$
for the measure  $e^{- N \beta  { \rm Tr} V(H)/2 }/Z$
is given by the explicit formula (c.f. \eqref{expli28})
\begin{equation}\label{01}
\mu^{(N)}_{\beta, V}(\lambda)\rd \lambda \sim e^{- \beta N \cH(\lambda)} \rd\lambda
 \quad \mbox{with Hamiltonian} \quad
\cH(\lambda) :=   \sum_{k=1}^N  \frac{1}{2}V(\lambda_k)- 
\frac{1}{N} \sum_{1\leq i<j\leq N}\log (\lambda_j-\lambda_i) ,
\end{equation}
where the parameter $\beta$ is determined by the symmetry class: 
 $\beta= 1$ for symmetric matrices,  $\beta= 2$ for Hermitian matrices and  $\beta = 4 $ for  
self dual quaternion matrices. The key structural ingredient
 of this formula, the Vandermonde determinant, 
is the same as in the Gaussian case, \eqref{expli28}. Thus all previous computations, 
developed for the Gaussian
case, can be carried out for $\beta=1, 2, 4$
provided that the Gaussian weight function  for the orthogonal polynomials 
is replaced with the function $e^{- \beta V(x) /2}$. 
Thus the analysis of the correlation functions depends 
critically  on the the asymptotic properties of the 
 corresponding orthogonal polynomials.
 In the pioneering  work 
of Dyson, Gaudin and Mehta , the potential 
is the quadratic polynomial $V(x) = x^2/2$ and the orthogonal polynomials are the Hermite polynomials
whose asymptotic properties are well-known.

The extension of  this approach to  a general potential is a 
demanding task;  important  progress was  made since  the late 1990's by 
Fokas-Its-Kitaev \cite{FIK},   Bleher-Its \cite{BI}, Deift {\it et. al.}
\cite{De1,  DKMVZ1, DKMVZ2},  Pastur-Shcherbina \cite{PS:97, PS} 
and more recently by
Lubinsky \cite{Lub}. { These results concern the simpler $\beta=2$ case.}
   For $\beta =1, 4$, the universality was established only quite recently 
for analytic $V$ with additional assumptions \cite{DG, DG1, KS, Sch} 
using earlier ideas of Widom  \cite{Wid}. 
 The final outcome  of these sophisticated  analyses is that  universality holds for the 
measure \eqref{01} in the sense 
that the short scale behavior of the correlation functions is independent of the potential
 $V$ (with appropriate assumptions) 
provided that  $\beta$ is one of the classical values, i.e., $\beta \in \{ 1,2,4\}$, that corresponds
to an underlying matrix ensemble.

 Notwithstanding  matrix ensembles or orthogonal polynomials,
the measure \eqref{01} is perfectly well defined for any $\beta>0$
and it can be interpreted as the Gibbs measure for a system of
particles with a logarithmic interaction (log-gas)
 at inverse temperature $\beta$. It is therefore a natural
question to extend universality  to  non-classical $\beta$
but the orthogonal polynomial methods are difficult to apply for this case.
  For all $\beta > 0$ the local statistics for the Gaussian case $V(x)=x^2/2$ 
can,  however, be characterized by the  ``Brownian carousel''  \cite{RRV, VV} which was
 derived from  a tridiagonal matrix representation
 \cite{DumEde} of Gaussian random matrices.

\bigskip

Apart from the invariant ensembles there are many natural non-invariant  
ensembles; the simplest  and most important  one being the {\it Wigner ensemble} for which 
the  matrix elements are  independent subject to a symmetry requirement,
e.g. $h_{ij}= \bar h_{ji}$
in the Hermitian case.
For non-invariant ensembles  there is no explicit formula analogous to \eqref{01} for
the joint distribution of the eigenvalues.  Hence the methods  for  the invariant
 ensembles described above are not applicable.
Until very recently, most rigorous results  
 have been on the density of eigenvalues, i.e. the convergence to the  the Wigner semicircle law \eqref{sc} 
was established with certain error estimates, see e.g. the works by Bai et al \cite{BMT}
and Guionnet and Zeitouni \cite{GZ}.
 The universality of the local statistics could only be
established for Hermitian Wigner matrices with a substantial Gaussian
component by Johansson \cite{J} and Ben Arous-P\'ech\'e \cite{BP}.
 All previous results on local universality
have relied on explicitly computable algebraic formulae. These were
provided by orthogonal polynomials
in case of the invariant ensembles, and by a modification of
 the Harish-Chandra/Itzykson/Zuber
integral in case of \cite{J}.  Nevertheless, following Wigner's thesis,
 universality is expected to hold for general Wigner matrices as well.

\medskip

Having summarized the existing rigorous results that were available 
until  2008, 
we set the two main problems we wish to address in this  article:

\medskip
\noindent
{\it Problem 1}: Prove the Wigner-Dyson-Gaudin-Mehta conjecture, i.e.
the universality for Wigner matrices
 with a general distribution for the matrix elements.

\medskip 
\noindent 
{\it Problem 2}:  Prove the universality of  the local statistics
for the log-gas \eqref{01} for all $\beta > 0$.

\medskip

We were able to  solve  Problem 1
 for a very general
class of distributions. As for  Problem 2, we solved it for the case of  real analytic   potentials $V$
 assuming that the equilibrium measure is supported on a single interval,
which, in particular, holds for any convex potential.
 We now state  our results precisely.

\begin{theorem}[Wigner-Dyson-Gaudin-Mehta conjecture]\cite[Theorem 7.2]{EKYY2} \label{bulkWigner}
Suppose that $H = (h_{ij})$ is a Hermitian (respectively, symmetric)  Wigner matrix.
 Suppose that for some $\e>0$  
\begin{equation} \label{4+e}
\E \left | \sqrt N  h_{ij}  \right | ^{4 + \e}  \;\leq\; C\,,
\end{equation}
for some constant $C$. 
Let $n \in \N$ and $O : \R^n \to \R$ be compactly supported and continuous.
Let $E$ satisfy $-2 < E < 2$ and let $\xi > 0$. 
Then for any sequence $b_N$ satisfying
 $N^{-1 + \xi} \leq b_N 
\leq \left | { |E|- 2}  \right |/2 $ we have
\begin{multline}\label{avg}
\lim_{N \to \infty} \int_{E - b_N}^{E + b_N} \frac{\rd E'}{2 b_N} \int_{\bR^n} \rd \alpha_1 \cdots
 \rd \alpha_n\, O(\alpha_1, 
\dots, \alpha_n) 
\\ 
{}\times{}  \frac{1}{\varrho_{sc}(E)^n} \left ( {p_{N}^{(n)} - p_{{\rm G}, N}^{(n)}} \right ) 
 \left ( {E' +
\frac{\alpha_1}{N\varrho_{sc}(E)}, \dots, E' + \frac{\alpha_n}{N\varrho_{sc}(E)}}\right )  \;=\; 0\,.
\end{multline}
Here $\varrho_{sc}$ is the semicircle law
 defined in \eqref{sc}, $p_N^{(n)}$ is the $n$-point correlation function
 of the eigenvalue 
distribution of $H$, and $p^{(n)}_{{\rm G},N}$ is the $n$-point correlation function
of an $N \times N$ GUE (respectively, GOE) matrix.
\end{theorem}

We remark that the convergence in this theorem is in weak sense,
and it also involves averaging over a small energy interval  $E'\in [E-b_N, E+b_N]$. 
Stronger types of convergence may also be considered and we will comment
on one possible such extension in Section~\ref{sec:put}. 
We believe that the issue of convergence types is of a technical nature and it is  dwarfed by
 the challenge to prove universality for the largest possible family of matrix ensembles. 
{\it  The fundamental challenge in random matrix theory remains in  answering the question of 
 why random matrix law is ubiquitous for seemingly disparate ensembles and physical systems. }
We will present a few extensions in this direction  in Sections~\ref{sec:gen} and
\ref{sec:ER}.

\medskip

In the case of  invariant ensembles, it is well-known that for 
 $V$ satisfying certain mild conditions the sequence
of one-point correlation functions, or
 densities, associated with  $\mu^{(N)}$ has a limit
as $N\to\infty$
and the limiting  equilibrium density  $\varrho(s)$
can be obtained as the unique minimizer  of the
functional
$$
I(\nu)=
\int_\bR V(t) \nu(t)\rd t-
\int_\bR \int_\bR \log|t-s| \nu(s) \nu(t) \rd t \rd s.
$$
Moreover, for convex $V$ the support of $\varrho$ is a single interval $[A,B]$
and $\varrho$ satisfies the equation 
\be
   \frac{1}{2}V'(t) = \int_\bR \frac{\varrho(s)\rd s}{t-s}
\label{equilibrium}
\ee
for any $t\in(A,B)$.  For the Gaussian case, $V(x)=x^2/2$, the equilibrium density
is given by the semicircle law $\varrho=\varrho_{sc}$, see \eqref{sc}.

\begin{theorem}[Bulk universality of $\beta$-ensemble]\cite[Corollary  2.2]{BEY} \label{bulkbeta}
 Assume
$V$ is a real analytic  function with 
$\inf_{x\in\R}V''(x) > 0$. 
Let $\beta> 0$.
 Consider the $\beta$-ensemble $\mu=\mu_{\beta, V}^{(N)}$ given in \eqref{01} 
and let
 $p_N^{(n)}$ denote the $n$-point correlation functions
of $\mu$, defined analogously to \eqref{pk}.
 For the Gaussian case, $V(x) =x^2/2$, the correlation
functions are denoted by $p_{G,N}^{(n)}$.
Let $E\in (A,B)$ lie in the interior of the support of $\varrho$
 and
similarly let $E'\in (-2,2)$ be inside the support of $\varrho_{sc}$. 
Let $O:\R^n\to \R$ be a smooth, compactly supported function.
Then for $b_N = N^{-1+\xi}$ with any $0<\xi \le 1/2$ we have
\begin{align}
\lim_{N \to \infty} \int  & \rd \alpha_1 \cdots \rd \alpha_n\, O(\alpha_1, 
\dots, \alpha_n) \Bigg [ 
   \int_{E - b_N}^{E + b_N} \frac{\rd x}{2 b_N}  \frac{1}{ \varrho (E)^n  }  p_{N}^{(n)}   \Big  ( x +
\frac{\alpha_1}{N\varrho(E)}, \dots,   x + \frac{\alpha_n}{N\varrho(E)}  \Big  ) \\ \nonumber
&
 -   \int_{E' - b_N}^{E' + b_N} \frac{\rd x}{2b_N}  \frac{1}{\varrho_{sc}(E')^n} p_{{\rm G}, N}^{(n)}      \Big  ( x +
\frac{\alpha_1}{N\varrho_{sc}(E')}, \dots,   x + \frac{\alpha_n}{N\varrho_{sc}(E')}  \Big  ) \Bigg ]
 \;=\; 0\,,
\end{align}
 i.e. the appropriately normalized  correlation functions  of  the
measure $\mu_{\beta, V}^{(N)}$ at the level $E$ in the bulk of the limiting density 
 asymptotically coincide with those
of the Gaussian case and they are independent of the value of $E$ in the bulk. 
\end{theorem}

We close this introduction with some short remarks concerning these
two theorems. 
Theorem \ref{bulkWigner} holds for a much larger class of 
 matrix ensembles with independent
entries and we will review 
some of them in Sections~\ref{sec:gen} and \ref{sec:ER}. 
Although Theorem \ref{bulkWigner} in its current form was proved in \cite{EKYY2},
the key ideas have been developed through several important 
steps in \cite{EPRSY, ESY4, EYY, EYY2, EYYrigi}. 
In particular, the Wigner-Dyson-Gaudin-Mehta (WDGM)  conjecture for Hermitian matrices was
 first solved  in \cite{EPRSY} in a joint work with the current authors and P\'ech\'e,  Ram\'irez and   Schlein.
This result holds  whenever the distributions of matrix elements are smooth.   The smoothness requirement was  
 partially removed in \cite{TV} and completely removed in a joint paper  with Ram\'irez,  Schlein, Tao and Vu
\cite{ERSTVY}. The  WDGM conjecture
for symmetric matrices was resolved  in \cite{ESY4}. In this paper,  
a novel idea based on Dyson Brownian motion was discovered. The most difficult case, 
the real symmetric Bernoulli matrices, was solved in \cite{EYY2} where a ``Fluctation Averaging 
 Lemma" (Lemma \ref{cancel} of the current paper)  
exploiting cancellation of matrix elements of the Green function was first introduced.
We will give a more detailed historical review  in Section~\ref{sec:hist}.

\bigskip
\noindent
The proof  of Theorem \ref{bulkWigner} 
   consists of  the following three steps,
discussed in Sections~\ref{sec:refined}, \ref{sec:DBM} and \ref{sec:4mom}, respectively.
Our  three-step  strategy was first introduced in \cite{EPRSY}.

\medskip 
\noindent
{\it Step 1.}  {\it Local semicircle law and delocalization of eigenvectors:}
It states that the density of eigenvalues
is  given by
the semicircle law not only as a weak limit on macroscopic scales \eqref{sc}, but
also in a strong sense and down to short scales containing 
only $N^\e$ eigenvalues for all $\e> 0$.  
This will imply the  {\it rigidity of eigenvalues}, 
 i.e., that   the eigenvalues are near their classical location
  in the sense to be made 
clear in Section~\ref{sec:DBM}. 
We also obtain precise estimates on the matrix elements of the Green function
 which in particular imply complete delocalization of eigenvectors.

\medskip 
\noindent
{\it Step 2.}
{\it universality for
Gaussian divisible ensembles:}  The Gaussian divisible ensembles are matrices of the form 
$ H_t= e^{-t/2} H_0+ \sqrt {1 - e^{-t}} U$,
 where $H_0$ is a Wigner matrix and
$U$ is an independent GUE matrix. The parametrization of  $H_t$ reflects that it is most conveniently obtained
by an Ornstein-Uhlenbeck  process.  There are two methods and both methods imply 
 the bulk  universality of $H_t$ for $t = N^{-\tau}$ for the entire  range of $0< \tau<1$ with  different estimates.

\begin{description}
\item[2a] {\it Proposition 3.1 of \cite{EPRSY} which uses an extension of Johansson's formula \cite{J}.}
\item[2b] {\it  Local ergodicity of the Dyson Brownian motion (DBM):}
\end{description}

The approach in 2a yields a slightly stronger estimate than 
the approach in 2b,  but it works only in the  Hermitian case.  
In this review, we will focus on the Dyson Brownian approach.

\medskip 
\noindent
{\it Step 3.}  {\it Approximation by Gaussian divisible ensembles:}
 It is a simple 
 density argument  in the space
of matrix ensembles which 
shows that for any probability distribution of the matrix
 elements there exists a Gaussian divisible distribution
with a small Gaussian component, as in Step 2, such 
that the two  associated Wigner ensembles
have asymptotically identical local eigenvalue statistics. 
 The first implementation  of this approximation scheme was via  a reverse heat flow 
argument  \cite{EPRSY}; it was later replaced by  the
  {\it Green function comparison theorem} \cite{EYY}.

\medskip

\noindent
The proof of Theorem \ref{bulkbeta}  consists of  the following two steps
that will be presented in Sections~\ref{beta} and \ref{sec:loceq}.

\medskip
\noindent
{\it Step 1.  Rigidity of eigenvalues.}  
This establishes that the location of the eigenvalues  are not too far
from their classical locations  determined by the equilibrium density $\varrho(s)$.

\medskip
\noindent
{\it Step 2.  Uniqueness of local Gibbs measures with logarithmic interactions.}
With the precision of eigenvalue location estimates from the Step 1 as an input,
the eigenvalue spacing distributions are
 shown to be  given by the corresponding 
Gaussian ones. (We will take the uniqueness of the spacing
distributions as our definition of the uniqueness of Gibbs state.)

\medskip 

There are several similarities and differences  between these two  methods. 
Both start with rigidity estimates on eigenvalues and then
 establish that the local spacing distributions are the same
 as in the Gaussian cases. 
The Gaussian divisible ensembles, which play a  key role in our theory for noninvariant ensembles, 
are  completely absent for invariant ensembles.
The key connection between the two methods, however,  
is the usage of DBM (or its analogue) in the Steps 2. In Section~\ref{sec:DBM}, we will 
first  present this idea.

The method for the proof of Theorem~\ref{bulkWigner} is extremely general. 
As of this writing, it has been applied to the generalized Wigner ensembles, the sample covariance
 ensembles and the Erd{\H o}s-R\'enyi  
matrices for certain range of the sparseness parameter.
It can also be extended to the  edges of the spectrum,
and it yields edge universality 
under more general conditions than  were previously known. 
This will be reviewed in Section~\ref{sec:edge}.
 Extensions to generalized Wigner matrices and Erd{\H o}s-R\'enyi  
matrices will also be discussed in Sections~\ref{sec:gen} and \ref{sec:ER}.
As the proof of Theorem \ref{bulkbeta} was just completed,  
we do not know how far this method can reach;  currently 
we can generalize the result to the nonconvex case 
under the assumption that the equilibrium measure $\rho$ is supported on a single interval
\cite{BEY2}.    
The theory we have developed to prove  Theorems~\ref{bulkWigner} and \ref{bulkbeta}
is purely analytic and we believe that it 
unveils the genuine mechanism of the Wigner-Dyson-Gaudin-Mehta universality.
 Finally, a short summary concerning   the recent history  
of universality is given in Section~\ref{sec:hist}.

\medskip

{\it Acknowledgement.} The results in this review were  obtained in collaboration
with Benjamin Schlein, Jun Yin, Antti Knowles and Paul Bourgade and in some work, also with
Jose Ramirez and Sandrine P\'eche.
 This article  is to report the joint progress 
with these authors.

\section{Dyson Brownian motion and the
local  relaxation flow}\label{sec:DBM}

\subsection{Concept and results}

The Dyson Brownian motion (DBM) describes the evolution of the eigenvalues
 of a Wigner matrix as an interacting point process
 if each matrix element $h_{ij}$ evolves according to independent
 (up to symmetry restriction) 
Brownian motions. We will slightly alter this definition by
 generating the dynamics of the matrix elements
by an Ornstein-Uhlenbeck (OU) process  which leaves the standard Gaussian
 distribution invariant. 
 In the Hermitian case, the OU process for the rescaled matrix elements 
 $v_{ij}: = N^{1/2}h_{ij}$ is given by the 
stochastic differential equation 
\be
  \rd v_{ij}= \rd \beta_{ij} - \frac{1}{2} v_{ij}\rd t, \qquad
i,j=1,2,\ldots N,
\label{zij}
\ee
where $ \beta_{ij}$,  $i <  j$, are independent complex Brownian
motions with variance one and $ \beta_{ii}$ are real 
Brownian motions of the same variance. 
 Denote the distribution of 
the eigenvalues $\lambda=(\lambda_1, \lambda_2,\ldots, \lambda_N)$
 of $H_t$ at  time $t$
by $f_t ({\mathbf \lambda})\mu_G (\rd {\bf \lambda})$
where $\mu_G$ is given by \eqref{01} with the potential $V(x) = x^2/2$.

Then $f_t=f_{t,N}$ satisfies \cite{DyB}
\be\label{dy}
\partial_{t} f_t =  \cL f_t,
\ee
where 
\be
\cL=\cL_N:=   \sum_{i=1}^N \frac{1}{2N}\partial_{i}^{2}  +\sum_{i=1}^N
\Bigg(- \frac{\beta}{4} \lambda_{i} +  \frac{\beta}{2N}\sum_{j\ne i}
\frac{1}{\lambda_i - \lambda_j}\Bigg) \partial_{i}, \quad 
\partial_i=\frac{\partial}{\partial\lambda_i}.
\label{L}
\ee
The parameter $\beta$  is chosen as follows: $\beta= 2$ for complex 
Hermitian matrices  and $\beta=1$ for symmetric real matrices.
Our formulation of the problem has already taken into account  
Dyson's observation  that the invariant measure for this dynamics is $\mu_G$. 
 A natural question regarding the DBM is how fast the dynamics reaches equilibrium.
  Dyson had already posed this question in 1962:

\medskip

\noindent 
{\bf Dyson's conjecture \cite{DyB}:}  The global equilibrium of DBM is
 reached in  time of order one 
and  the  local equilibrium (in the bulk)  is reached in  time of order  $1/N$. 
  Dyson further remarked,

\medskip
\begin{minipage}[c]{6in}
{\it ``The picture of the gas coming into
equilibrium in two well-separated stages, with microscopic
and macroscopic time scales, is  suggested
 with the help of physical intuition. A
rigorous proof that this picture is accurate would
require a much deeper mathematical analysis."}
\end{minipage}

\bigskip

We will prove that Dyson's conjecture is correct if the 
initial data of the flow is a Wigner ensemble, which was 
Dyson's  original interest. Our result in fact is
 valid for DBM  with much more general initial data that we now survey. 
Briefly, 
it will turn out that the {\it global} equilibrium is
indeed reached within a time of order one,
but {\it local} equilibrium is achieved much faster if  an a-priori
estimate  on the location of the eigenvalues (also called points) is satisfied. 
To formulate this estimate,  
let $\gamma_j =\gamma_{j,N}$ denote  the location of the $j$-th point
under the semicircle law, i.e., $\gamma_j$ is defined by
\be\label{def:gamma}
 N \int_{-\infty}^{\gamma_j} \varrho_{sc}(x) \rd x = j, \qquad 1\leq j\le N. 
\ee
We will call $\gamma_j$ the {\it classical location} of the $j$-th point.

\bigskip
\noindent 
{\bf A-priori Estimate:} There exists an ${\frak a}>0$ such that
\be
 Q=Q_\fa:= \sup_{t\ge N^{- 2 { \frak a}}}   \frac{1}{N}
 \int \sum_{j=1}^N(\lambda_j-\gamma_j)^2
 f_t( \lambda )\mu_G(\rd \lambda) \le CN^{-1-2{\frak a}}
\label{assum3}
\ee
with a constant $C$ uniformly in $N$.
(This a-priori estimate was referred to as Assumption III
 in \cite{ESY4, ESYY}.)

The main result on the local ergodicity of Dyson Brownian motion states
 that if the a-priori estimate \eqref{assum3} is satisfied  
then the local correlation functions  of the measure  $f_t \mu_G$ 
are the same as the corresponding ones for the Gaussian measure, 
 $\mu_G = f_\infty\mu_G$,
 provided that
 $t$ is larger than $N^{-2{\frak a}}$.
The $n$-point
correlation functions of the probability measure $f_t\rd\mu_G$ are defined,
similarly to \eqref{pk}, by
\be
 p^{(n)}_{t,N}(x_1, x_2, \ldots,  x_n) = \int_{\R^{N-n}}
f_t(\bx) \mu_G(\bx) \rd x_{n+1}\ldots
\rd x_N, \qquad  \bx=(x_1, x_2,\ldots, x_N).
\label{corr}
\ee
Due to the convention that one can view the locations of eigenvalues as the coordinates of particles,  we have used $\bx$, 
instead of $\bl$,  in the last equation. From now on, we will use both conventions depending 
on which viewpoint we wish to emphasize. 
Notice that the probability distribution  of the eigenvalues at the 
time $t$, $f_t \mu_G$, 
 is the same as that of the Gaussian divisible  matrix: 
\be\label{matrixdbm}
H_t = e^{-t/2} H_0 + (1-e^{-t})^{1/2}\, U,
\ee
where $H_0$  is the initial  Wigner matrix and 
$U$ is an independent standard GUE (or GOE) matrix. This establishes  the universality of 
the Gaussian divisible ensembles. The precise statement is the following
theorem:

\begin{theorem}\label{thm:DBM} \cite[Theorem 2.1]{ESYY} Suppose that
 the a-priori estimate \eqref{assum3}  holds
 for the solution $f_t$ of the forward equation \eqref{dy} with some exponent $\fa>0$.
 Let $E\in (-2, 2) $ and $b>0$
such that $[E-b,E+b] \subset (-2, 2)$. Then for any $s>0$, 
for any integer $n\ge 1$ and for any compactly supported continuous test function
$O:\bR^n\to \bR$, we have
\be
\begin{split}
\lim_{N\to \infty} \sup_{t\ge  N^{-2\fa+s} } \;
\int_{E-b}^{E+b}\frac{\rd E'}{2b}
\int_{\R^n} &  \rd\alpha_1
\ldots \rd\alpha_n \; O(\alpha_1,\ldots,\alpha_n) \\
&\times \frac{1}{\varrho_{sc}(E)^n} \Big ( p_{t,N}^{(n)}  - p_{G, N} ^{(n)} \Big )
\Big (E'+\frac{\alpha_1}{N\varrho_{sc}(E)},
\ldots, E'+\frac{\alpha_n}{ N\varrho_{sc}(E)}\Big) =0.
\label{abstrthm}
\end{split}
\ee
\end{theorem}

We can choose  $b=b_N$ depending on $N$.  In \cite{ESYY} explicit bounds on the
  speed of convergence and the optimal range of $b$  were also established.  
In particular, thanks to the 
optimal rigidity estimate  \cite{EYYrigi}, i.e., \eqref{assum3} with $\fa =1/2$, the
range of the energy averaging  in \eqref{abstrthm}  was reduced
to $b_N\ge N^{-1+\xi}$, $\xi>0$, 
 but only for $t\ge N^{-\xi/8}$ (Theorem 2.3  of \cite{EYYrigi}).

Theorem \ref{thm:DBM} is a consequence of the following theorem which identifies the gap distribution
of the eigenvalues. 

\begin{theorem}[Universality of the Dyson Brownian motion for short time]\label{thmM}
\cite[Theorem 4.1]{ESYY} \\
Suppose  $\beta\ge1$ and 
let  $G:\bR \to\bR$ be a  smooth function with compact 
support. 
Then for 
any sufficiently small $\e>0$, independent of $N$, there exist
constants $C, c>0$, depending only on $\e$ and $G$ such that
 for any $J\subset \{ 1, 2, \ldots , N-1\}$  
  we have
\be\label{GG}
\Big| \int \frac 1 {|J|} \sum_{i\in J} G( N(x_i-x_{i+1})) f_t  \rd \mu_G -
\int \frac 1 {|J|} \sum_{i\in J} G( N(x_i-x_{i+1})) \rd\mu_G \Big|
\le C N^{\e} \sqrt{  \frac {N^2 Q} { |J| t }}  + Ce^{-c N^\e } .
\ee
In particular,  if the a-priori estimate \eqref{assum3} holds with some
$\fa>0$ and $|J|$ is of order $N$, then for any  $t > N^{-2 \fa + 3\e}$
the right hand side converges to zero as $N\to\infty$, i.e. the 
gap distributions for $f_t\rd\mu_G$ and $\rd\mu_G$ coincide.
\end{theorem}

The test functions can be generalized to 
\be
G\Big( N(x_i-x_{i+1}), N(x_{i+1}-x_{i+2}), \ldots, N(x_{i+{n-1}}-x_{i+n})\Big)
\label{cG}
\ee
for any $n$ fixed  which is needed to identify higher order correlation functions.
 In applications, $J$ is chosen to be the
 indices of the eigenvalues in the interval $[E-b, E+ b]$
and thus $|J| \sim N b$. 
This identifies the gap distributions of eigenvalues completely  and thus also identifies 
the correlation functions and concludes Theorem \ref{thm:DBM}.
Note that the input of this theorem, the apriori estimate \eqref{assum3}, identifies
the location of the eigenvalues only on a scale $N^{-1/2-\fa}$ which is much weaker 
than the $1/N$ precision for the eigenvalue differences in \eqref{GG}.

 By the rigidity estimates (see Corollary \ref{7.1} below),
  the a-priori estimate \eqref{assum3} holds for any $\fa < 1/2$
if the initial data of the DBM is a Wigner ensemble.  
Therefore, Theorem \ref{thmM}  holds for any $t\ge N^{-1 + \e}$ 
for any $\e > 0$ and {\it this establishes Dyson's  conjecture. }

\subsection{Main ideas behind the proof of Theorem \ref{thmM}}

The key method is to  analyze the relaxation to equilibrium of the dynamics \eqref{dy}.
This approach  was first introduced  in Section 5.1 of \cite{ESY4};
the presentation here follows \cite{ESYY}.

 We start with a short review of
the logarithmic Sobolev inequality for a general measure.
Let the probability measure $\mu$ on $\bR^N$ be given by  a general 
Hamiltonian $\cH$:
\be
  \rd\mu(\bx) = \frac{e^{-N \cH(\bx)}}{Z}\rd \bx, \qquad 
\label{convBE}
\ee
and let $\cL$ be the generator, symmetric with respect to the measure $\rd\mu$,
 defined by the associated   Dirichlet form
\be\label{D}
   D(f)= D_\mu(f) = -\int f \cL f \rd \mu := \frac{1}{2N}\sum_j 
  \int(\pt_j f)^2 \rd \mu, \qquad \pt_j = \pt_{x_j}.
\ee
Recall  the relative entropy of two  probability measures:
$$
   S (\nu| \mu ): =   
\int   \frac {\rd \nu}  {\rd \mu} \log \left ( \frac {\rd \nu}  {\rd \mu} \right ) \rd\mu.
$$
If $\rd\nu = f\rd\mu$, then we will sometimes use the notation $S_\mu(f):= S(f\mu|\mu)$.
The entropy can be used to control the total variation norm via the well known inequality 
\be
   \int |f-1|\rd \mu \le \sqrt{ 2 S_\mu(f)}.
\label{entropyneq}
\ee

Let $f_t$ be the solution to  the evolution equation
\be
  \pt_t f_t = \cL f_t,  \qquad t>0,
\label{dybe}
\ee
with a given initial condition $f_0$. The
evolution of the entropy $S_\mu(f_t)=S(f_t \mu | \mu )$
satisfies
\be
  \pt_t S_\mu(f_t )  = -4 D_\mu (\sqrt{f_t}).
\label{derS}
\ee
By Bakry and \'Emery \cite{BakEme},  the evolution of the Dirichlet form
satisfies  the inequality 
\be\label{eq:BE}
   \pt_t D_\mu(\sqrt{f_t}) 
  \le   -  \frac{1}{2 N} \int
 (\nabla \sqrt f_t)(\nabla^2\cH)\nabla \sqrt  f_t \rd \mu.
\ee
If the Hamiltonian is convex, i.e., 
\be\label{convexham}
 \nabla^2\cH(\bx)=\mbox{Hess} \, \cH(\bx) \ge \varpi 
 \qquad \mbox{for all $\bx\in \bR^N$}
\ee 
 with some constant $\varpi>0$,
then  we have
\be
    \pt_t D_\mu (\sqrt{f_t}) \le - \varpi D_\mu(\sqrt{f_t}).
\label{derD}
\ee
Integrating \eqref{derS} and \eqref{derD} back from infinity to 0,
we obtain
the  {\it logarithmic Sobolev inequality (LSI)}
\be
 S_\mu(f) \le \frac{4}{\varpi}  D_\mu (\sqrt{f}), \quad f = f_0
\label{lsi1}
\ee
and the  {\it exponential relaxation
of the entropy and Dirichlet form on time scale $t\sim 1/\varpi$}
\be
    S_\mu(f_t )  \le e^{-t \varpi } S_\mu(f_0  ),
 \quad D_\mu (\sqrt{f_t}) \le e^{-t \varpi } D_\mu (\sqrt{f_0}) .
\label{Sdec}
\ee
As a consequence of the logarithmic Sobolev inequality, 
we also have the concentration inequality 
for any $k$ and $a>0$
\be\label{concen}
\int {\bf 1}\left(|x_k-\E_\mu(x_k)|> a\right)\rd\mu\leq 2e^{-\varpi N a^2/2 }.
\ee
We will not use this inequality in this section, but it will become important in Section \ref{beta}.

 Returning to the  classical ensembles, we  assume from now on
that $\cH$ is given by \eqref{01}  with $V(x)=x^2/2$ and the
equilibrium measure is the Gaussian one, $\mu=\mu_G$.
We then have  the convexity inequality
\be\label{convex}
\Big\langle \bv , \nabla^2  \cH(\bx)\bv\Big\rangle
\ge   \frac{1}{2} \,  \|\bv\|^2 + \frac{1}{N}
 \sum_{i<j} \frac{(v_i - v_j)^2}{(x_i-x_j)^2} \ge  \frac{1 }{2} \,  \|\bv\|^2,
 \qquad \bv\in\bR^N.
\ee
This  guarantees that $\mu$ satisfies the LSI with $\varpi=1/2$ and the relaxation time to 
equilibrium is of order one.

The key idea is that the  relaxation time is in fact
much shorter than order one for local observables that
depend only on the eigenvalue {\it differences}. Equation \eqref{convex} shows that 
the relaxation in the direction $v_i-v_j$ is much faster than order
 one provided that $x_i- x_j$ are close.
However, this effect is hard to exploit directly  due to that all modes of different 
wavelengths are coupled.
Our idea is to add an auxiliary {strongly convex}  potential $W(\bx)$ to the Hamiltonian
to  ``speed up''
the convergence to local equilibrium. On the other hand, we will also 
show that the cost of this  speeding up 
can be effectively controlled if the a-priori estimate \eqref{assum3} holds.

The auxiliary  potential $W(\bx)$ is defined by 
\be
    W(\bx): =    \sum_{j=1}^N
W_j (x_j)   , \qquad W_j (x) := \frac{1}{2 \tau } (x_j -\gamma_j)^2,
\label{defW}
\ee
i.e. it is a quadratic confinement on scale $\sqrt \tau $ for each eigenvalue
near its classical location, where the parameter $\tau>0$ will be chosen later. 
The total Hamiltonian is given by 
\be
\wt \cH: = \cH +W,
\label{def:wth}
\ee
where $\cH$ is the Gaussian Hamiltonian given by \eqref{01}.
The measure with Hamiltonian $\wt \cH$, 
$$
   \rd\om: =\om (\bx)\rd\bx, \quad \om: = e^{-N\wt\cH}/\wt Z,
$$
will be called  the {\it local relaxation  measure}.   This measure was named the  {\it pseudo-equilibrium measure} 
in our previous papers.

The  {\it local relaxation flow}  is defined to be the flow with 
the  generator characterized by  the natural Dirichlet form w.r.t. $\om$, explicitly,  
$\wt \cL$:
\be\label{tl}
\wt \cL = \cL - \sum_j b_j \partial_j, \quad
b_j = W_j'(x_j)= \frac{x_j -\gamma_j}{\tau}.
\ee
 We will typically choose $\tau \ll 1$ 
 so that the additional term $W$  substantially
increases the lower bound \eqref{convexham} on the Hessian, hence speeding up the
dynamics so that the relaxation  time is at most  $\tau$.

\medskip

 The idea of adding an artificial potential $W$ to speed up the convergence appears to be unnatural 
here. The current formulation is a {streamlined version} of a much more complicated approach
that   appeared in 
\cite{ESY4} and which 
 took ideas from the earlier work \cite{ERSY}.  Roughly speaking, 
in hydrodynamical limit, 
the  short wavelength modes always have shorter  relaxation times than the long wavelength modes.
A direct implementation of this idea is extremely complicated 
due to the logarithmic interaction that couples short and long wavelength modes.
 Adding a strongly convex  auxiliary  potential $W(\bx)$
shortens the relaxation time of  the long wavelength modes, 
but it does not affect the short modes, i.e. the local statistics, which are our main interest. 
The analysis of the new system is much simpler since now the relaxation is faster, uniform for
all modes. Finally, we need to compare the local statistics of the original system
with those of the modified one. It turns out that the difference is governed
by  $(\nabla W)^2$ which can be directly controlled  by the 
a-priori estimate \eqref{assum3}.

Our method for enhancing the convexity of $\cH$ is reminiscent of a standard
convexification idea concerning metastable states.
To explain the similarity, consider a particle near one of the
local minima of a double well potential separated by a local maximum,
or energy barrier. Although the potential is not convex globally,
one may still study a reference problem defined by convexifying
 the potential along with the  well in which the particle initially resides.  Before the 
particle reaches the energy barrier, there is no difference between 
these two problems. Thus questions concerning time scales
shorter than the typical escape time can be conveniently answered 
by considering the convexified problem; in particular the
escape time in the metastability problem itself can be estimated
by using convex analysis.
Our DBM problem  is already convex, but not sufficiently convex.
The modification by adding $W$ enhances convexity without altering 
the local statistics. This is similar to  the convexification in the metastability
problem which does not alter events before the escape time.

\subsection{Some details on the proof of  Theorem \ref{thmM}}

The core of the proof is divided into three theorems.
For the flow with generator $\wt \cL$, 
we have the following estimates on the entropy and Dirichlet form.

\begin{theorem}\label{thm2} 
Consider the forward equation
\be
\partial_t q_t=\wt \cL q_t, \qquad t\ge 0,
\label{dytilde}
\ee
with initial condition $q_0=q$ and
with the reversible  measure $\omega$. Assume that   $\int q_0 \rd \om=1$.
Then we have the following estimates
\be\label{0.1}
\partial_t D_{\omega}( \sqrt {q_t}) \le - \frac{1}{2\tau}D_{\omega}( \sqrt {q_t}) -
\frac{1}{2N^2}  \int   \sum_{i,j=1}^N
\frac{  ( \pt_i \sqrt{ q_t} - \pt_j\sqrt {q_t} )^2}{(x_i-x_j)^2} \rd \omega ,
\ee
\be\label{0.2}
\frac{1}{2N^2} \int_0^\infty  \rd s  \int    \sum_{ i,j=1}^N
\frac{(\pt_i\sqrt {q_s} - \pt_j\sqrt {q_s} )^2 }{(x_i-x_j)^2}\rd \omega
\le D_{\omega}( \sqrt {q})
\ee
and the logarithmic Sobolev inequality
\be\label{lsi}
S_\om(q)\le C \tau  D_{\omega}( \sqrt {q})
\ee
with a universal constant $C$.
Thus the relaxation time to equilibrium is of order $\tau$:
\be\label{Sdecay}
 S_{\omega}(q_t)\le e^{-Ct/\tau} S_\omega(q).
\ee
\end{theorem}

\begin{proof} Denote by $h=\sqrt{q}$ and we have the equation 
\be\label{1.5}
\pt_t D_\om ( h_t) =   \pt_t \frac{1}{2N}\int 
(\nabla h)^2 e^{- N \wt\cH} \rd \bx
  \le  - \frac{1}{2 N }\int
\nabla h (\nabla^2 \wt\cH)\nabla h  e^{- N \wt\cH} \rd\bx.
\ee
In our case, \eqref{convex} and \eqref{defW} imply  that the Hessian of $\wt \cH$ is bounded
from below as
\be\label{convex4}
\nabla h (\nabla^2 \wt\cH)\nabla h
\ge        \frac C { \tau} \sum_j  (\partial_j h)^2 +
\frac{1}{2N   }  \sum_{ i,j}  \frac 1 {(x_i - x_j)^2} (\pt_i h -
\pt_j h)^2
\ee
with some positive constant $C$.
This proves \eqref{0.1} and \eqref{0.2}. The  rest can be proved by straightforward arguments
given in the earlier part of this section. 
\end{proof}

\medskip

The estimate \eqref{0.2}
plays a key role in the next theorem.

\begin{theorem}[Dirichlet form inequality]\label{thm3}
Let $q$ be a probability density $\int q\rd\om=1$ and
let  $G:\bR\to\bR$ be a  smooth function with compact support.
Then  for any $J\subset \{ 1, 2, \ldots , N-1\}$ and any $t>0$ we have
\be\label{diff}
\Big| \int \frac 1 {|J|} \sum_{i\in J} G(N(x_i - x_{i+1})) q \rd \omega -
\int \frac 1 {|J|} \sum_{i\in J} G(N(x_i - x_{i+1}) ) \rd \omega \Big|
\le C \Big( t  \frac { D_\omega (\sqrt {q})  }{|J|}  \Big)^{1/2}  + C \sqrt
{S_\om(q)} e^{-c t/\tau} .
\ee
\end{theorem}

{\it Proof.} For simplicity, we  assume that $J =  \{ 1, 2, \ldots , N-1\}$. 
Let $q_t$ satisfy
\[
\partial_t q_t = \wt \cL q_t, \qquad t\ge 0,
\]
with an initial condition $q_0=q$.
We write
\begin{align}
\int \Big[  \frac 1 {|J|} & \sum_{i\in J} G(N(x_i - x_{i+1}))\Big] (q-1)\rd\om \nonumber\\
&  =  \int \Big[  \frac 1 {|J|} \sum_{i\in J}G(N(x_i - x_{i+1}))\Big] (q-q_t)\rd\om
 +   \int \Big[  \frac 1 {|J|} \sum_{i\in J} G(N(x_i - x_{i+1})) \Big] (q_t-1)\rd\om.
\label{splitt}
\end{align}
The second term can be estimated by  \eqref{entropyneq},
 the decay of the entropy \eqref{Sdecay} and the boundedness of $G$; 
this gives the second term in \eqref{diff}.

To estimate the first term in \eqref{splitt},
 by the evolution equation  $\pt q_t =\wt \cL q_t$ and the definition of $\wt \cL$: 
\begin{align}
\int  \frac 1 {|J|} \sum_{i\in J} &G(N(x_i - x_{i+1}))q_t \rd \omega  -
\int  \frac 1 {|J|} \sum_{i\in J} G(N(x_i - x_{i+1}))q_0 \rd \omega \non\\
&= \int_0^t  \rd s \int   \frac 1 {|J|} \sum_{i\in J}
   G'( N(x_i-x_{i+1}))
[\pt_{i} q_s - \pt_{i+1}q_s]  \rd \omega. \non
\end{align}
{F}rom the Schwarz inequality and $\pt q = 2 \sqrt{q}\pt\sqrt{q}$,
the last term is bounded by
\begin{align}\label{4.1}
2 & \left [   \int_0^t  \rd s \int_{\bR^N} 
 \frac {N^2} {|J|^2} \sum_{i\in J} \Big[ G' (N(x_i - x_{i +1}))\Big] ^2 
(x_{i}-x_{i+1})^2  \, q_s \rd \omega
\right ]^{1/2} \nonumber \\
 &\times \left [ \int_0^t  \rd s \int_{\bR^N}  \frac 1 {N^2 } \sum_i
\frac{1}{(x_{i}-x_{i+1})^2}  [ \pt_{i}\sqrt {q_s} -
\pt_{i+1}\sqrt {q_s}]^2  \rd \omega \right ]^{1/2} \nonumber \\
\le &  \; C \Big( \frac{D_\omega(\sqrt {q_0}) t}{|J|}\Big)^{1/2},
\end{align}
where we have used \eqref{0.2} and that
$ \Big[ G' (N(x_i - x_{i +1}))\Big]^2
 (x_{i}-x_{i+1})^2 \le CN^{-2}$
due to  $G$ being smooth and  compactly
supported.
\qed

\medskip

Alternatively, we could have directly estimated  the left hand side 
of \eqref{diff} by using the total variation norm between $q \om$ and $\om$, 
which  in turn could be estimated by the entropy 
  \eqref{entropyneq} and the Dirichlet form using the logarithmic Sobolev
inequality,  i.e., 
by 
\be\label{entbound}
C \int | q - 1| \rd\om \le  C \sqrt { S_\om(q) } \le C \sqrt { \tau D_\om (\sqrt q)}.
\ee
However, compared with this simple bound, the estimate
  \eqref{diff} gains an extra factor $|J|\sim N$ in the denominator,
 i.e. it is in terms of Dirichlet form {\it per particle}.
The improvement
 is due to the observable in \eqref{diff} being of special form and we exploit the 
term \eqref{0.2}.

\bigskip

The final  ingredient in proving Theorem \ref{thmM} is the following entropy and
Dirichlet form estimates.

\begin{theorem}\label{thm1}
Suppose that  \eqref{convex} holds. Let $\fa>0$ be fixed and
recall the definition of $Q=Q_\fa$ from \eqref{assum3}.
Fix a constant $\tau \ge  N^{-2 \fa}$ and consider the local relaxation 
measure $\omega$ with this $\tau$. 
Set $\psi:=\om/\mu$ and
let  $g_t: = f_t/\psi$.
Suppose there is a constant $m$ such that 
\be\label{entA}
S(f_{\tau} \om | \om )\le CN^m.
\ee
Then for any $ t \ge \tau N^\e$
the entropy and the Dirichlet form satisfy the estimates:
\be\label{1.3}
S(g_t \omega | \omega) \le
 C   N^2    Q \tau^{-1}, \qquad
D_\omega (\sqrt{g_t})
\le CN^2  Q \tau^{-2}
\ee
 where the constants depend on $\e$ and $m$.
\end{theorem}

{\it Proof.}  
The evolution of the entropy $S(f_t \mu |\om)= S_{\omega} (g_t) $
 can be computed explicitly by the formula \cite{Y}
$$
\partial_t  S(f_t \mu |\om)  = -  \frac{2}{N}  \sum_{j} \int (\partial_j
\sqrt {g_t})^2  \, \psi \, \rd\mu
+\int g_t  \cL \psi \, \rd\mu.
$$
Hence
we have, by using \eqref{tl},
$$
\pt_t S(f_t \mu |\om)
= - \frac{ 2}{N}  \sum_{j} \int (\partial_j
\sqrt {g_t})^2  \, \rd\omega
+\int   \wt \cL g_t  \, \rd\omega+  \sum_{j} \int  b_j \partial_j
g_t   \, \rd\omega.
$$
Since $\om$ is $\wt \cL$-invariant and time independent,
 the middle term on the right hand side vanishes, and
from the Schwarz inequality
\be\label{1.1}
 \pt_t S(f_t \mu |\om) \le  -D_{\omega} (\sqrt {g_t})
+   C  N\sum_{j} \int  b_j^2  g_t   \, \rd\omega \le
-D_{\omega} (\sqrt {g_t})
+   C  N^2 Q \tau^{-2} .
\ee
Together with the logarithmic Sobolev inequality  \eqref{lsi},  we have
\be\label{1.2new}
\partial_t S(f_t \mu |\om)\le  -D_{\omega} (\sqrt {g_t})
+   C  N^2 Q \tau^{-2}  \le  - C\tau^{-1}  S(f_t \mu |\om)+
 C  N^2  Q \tau^{-2}.
\ee
Integrating the last 
inequality  from $\tau$ to $t$ and using the assumption \eqref{entA}
and $t\ge \tau N^\e$, 
we have proved the first inequality of \eqref{1.3}. 
Using this result and integrating \eqref{1.1}, we have 
$$
\int_\tau^t  D_{\omega} (\sqrt {g_s})  \rd s  \le  C   N^2    Q \tau^{-1}. 
$$
By the convexity of the Hamiltonian, $ D_{\mu} (\sqrt {f_t}) $  is decreasing  in $t$. 
Since $ D_{\omega} (\sqrt {g_s}) \le  CD_{\mu} (\sqrt {f_s}) +  CN^2 Q \tau^{-2}$,
 this proves 
the second  inequality of \eqref{1.3}.
\qed

\bigskip
Finally, we complete the proof of Theorem \ref{thmM}. For any given $t>0$ we
 now choose 
$\tau: =  t N^{-\e}$ and we construct the local relaxation measure $\om$
with this $\tau$. Set $\psi= \om/\mu$ and
let $q:= g_{t}= f_{t}/\psi$  be the density $q$ in
Theorem \ref{thm3}.
 Then  Theorem \ref{thm1}, Theorem \ref{thm3}
 and an easy bound on the entropy $S_\om(q)\le CN^m$
imply  that
\be\label{diff1}
\Big| \int \frac 1 N \sum_{i\in J} G(N(x_i - x_{i+1}))(   f_t \rd \mu -
  \rd \omega) \Big| 
\le C \Big( t  \frac { D_\omega (\sqrt {q})  }{|J|}  \Big)^{1/2}  + C \sqrt
{S_\om(q)} e^{-c N^\e} .
\ee
\be\nonumber
\le C \Big( t  \frac { N^2 Q  }{|J| \tau^2 }  \Big)^{1/2}+  Ce^{-c N^{\e}}
    \le 
C N^{\e} \sqrt{  \frac {N^2 Q} { |J| t }}
+  Ce^{-c N^{\e}} , 
\ee
i.e.,  the local statistics of $f_t \mu$ and $\om$ are the same for any
 initial data  $f_\tau$ for which  
\eqref{entA} is satisfied. {Applying the same argument to 
the Gaussian initial data, $f_0=f_\tau=1$,
we can also compare $\mu$ and $\om$.}
   We have thus  proved \eqref{GG} and hence the 
universality. 
\qed

\section{Local semicircle law via Green function}\label{sec:refined}

The Wigner semicircle law asserts that \eqref{sc} is valid in  a weak limit,
 i.e., for any smooth test function $O$ with compact support 
we have 
\be
\E \int_{\bR} O(x) \left [ \varrho_N (x) -  \varrho_{sc}(x)\right] \rd x  \to 0.
\ee
This means that the density of eigenvalues in a window independent of $N$ is given by the 
semicircle law. Our goal is to prove a local version of this result 
for  windows  slightly larger than $1/N$ and in a large deviation sense.  The main object to study 
is the Green function of the matrix $G(z) = [H-z]^{-1}$,   $z=E+i\eta$, $E\in \R$, $\eta>0$, 
which is related to the  
 {\it Stieltjes transform} of the empirical measure:
\be
   m(z)= m_N(z): = \frac{1}{N}\tr \frac{1}{H-z} = \frac{1}{N}\sum_{j=1}^N
  \frac{1}{\lambda_j-z}=\int_\R \frac{\rd\varrho_N(x)}{x-z} = \frac 1 N \sum_{j=1}^N G_{jj} (z) .
\label{St}
\ee
We will compare it  with $m_{sc}(z): = \int_\R (x-z)^{-1}\varrho_{sc}(x)\rd x$,
 the Stieltjes transform of the semicircle law.
This is the content of the local semicircle law,  Theorem \ref{lsc} below.
The key parameter is  $\eta=\im z$ which determines the resolution, i.e.
the scale on which the local semicircle law holds.

For the rest of this paper, we will assume that the probability
distribution of the matrix elements satisfy 
the following subexponential condition:
\be\label{subexp}
   \P \big( |v_{ij}|\ge x )\le C_0\exp\big(-x^\ttau\big), \qquad x>0, 
\ee
with some positive constants $C_0, \ttau$,
where we set $v_{ij}=\sqrt{N}h_{ij}$.
This condition can be relaxed to \eqref{4+e} via a cutoff argument,
 but we will not discuss such technical details here.

\begin{theorem}[Local semicircle law]\cite[Theorem 2.1]{EYYrigi}\label{lsc}{}
Let $H=(h_{ij})$ be a Hermitian or symmetric $N\times N$ 
random matrix
with $\E\, h_{ij}=0$, $1\leq i,j\leq N$.  
Suppose that the distributions of the matrix elements have a uniformly 
  subexponential decay \eqref{subexp}.
Then  there exist positive constants  
 $A_0 > 1$,  $ C, c$ and $\phi < 1$   such that
with    
\be
L:= A_0\log\log N
\label{Lbound}
\ee
 the following estimates hold for any sufficiently large $N\ge N_0(C_0,\ttau)$:

(i) The Stieltjes transform of the empirical 
eigenvalue distribution of  $H $  satisfies 
\be\label{Lambdafinal} 
\P \Big ( \bigcup_{z\in \bS_L} \Big\{ |m(z)-m_{sc}(z)| 
 \ge \frac{(\log N)^{4L}}{N\eta} \Big\}   \Big )\le  C\exp{\big[-c(\log N)^{\phi L} \big]}, 
\ee
where
\be
{\bf  S}_L:=\Big\{ z=E+i\eta\; : \;
 |E|\leq 5,  \quad  N^{-1}(\log N)^{10L} < \eta \le  10  \Big\}.
\label{defS}
\ee

(ii) The individual  matrix elements of
the Green function  satisfy
\be\label{Lambdaodfinal}
\P \left  ( \bigcup_{z\in \bS_L} \left\{ \max_{i,j}\big|G_{ij}(z)-\delta_{ij}m_{sc}(z)\big|
 \geq 
(\log N)^{4L} \sqrt{\frac{\im m_{sc}(z)  }{N\eta}} + \frac{(\log N)^{4L}}{N\eta}
   \right\}    \right)
\leq  C\exp{\big[-c(\log N)^{\phi L} \big]}.
\ee

\end{theorem}

Theorem~\ref{lsc} is the strongest form of the local semicircle law
that gives optimal error estimates (modulo logarithmic factors)
on the smallest possible scale, which is valid uniformly in the
spectrum including the edge, and which controls not only the Stieltjes
transform but also individual  matrix elements of the resolvent.
This theorem is the final result of subsequent improvements
in \cite{ESY2, ESY3, EYY, EYY2, EYYrigi}
of our first local semicirle law in \cite{ESY1}.

\medskip

The local semicircle estimates imply that the $j$-th eigenvalue,  $\lambda_j$,  is very close to
its classical location $\gamma_j$, defined in \eqref{def:gamma}:
\begin{corollary} [Rigidity of eigenvalues] \cite[Theorem 2.2]{EYYrigi}\label{7.1}
Under the assumptions of Theorem~\ref{lsc}  we have
\be\label{rigidity}
\P \Bigg\{  \exists j\; : \; |\lambda_j-\gamma_j| 
\ge (\log N)^{ L}  \Big [ \min \big ( \, j ,  N-j+1 \,  \big) \Big  ]^{-1/3}   N^{-2/3} \Bigg\}
 \le  C\exp{\big[-c(\log N)^{\phi L} \big]}
\ee
for any sufficiently large $N\ge N_0$. 
\end{corollary}

This corollary in particular  proves the a-priori estimate \eqref{assum3} for any $\fa < 1/2$. 

Corollary \ref{7.1} is a simple consequence of the Helffer-Sj\"ostrand formula which 
  translates information on the  Stieltjes
transform of the empirical measure first to the counting function and
then to the locations of eigenvalues. 
The formula yields the representation 
\be
   f(\lambda) =\frac{1}{2\pi}\int_{\bR^2}
\frac{\partial_{\bar z} \wt f(x+iy)}{\lambda-x-iy} \rd x  \rd y 
=\frac{1}{2\pi}\int_{\bR^2}
\frac{iy f''(x)\chi(y) +i(f(x) + iyf'(x))\chi'(y) }{\lambda-x-iy} \rd x 
\rd y
\label{1*}
\ee
for any real valued $C^2$ function $f$
on $\R$, where $\chi(y)$ is any smooth cutoff function with bounded
derivatives and   
 supported in $[-1,1]$ with $\chi(y)=1$ for  $|y|\leq 1/2$. 
In the applications, $f$ will be a smoothed version of the
characteristic functions of spectral intervals so that $\sum_j f(\lambda_j)$
 counts eigenvalues in that interval.
The details of the argument can be found in  \cite{ERSY}.

\medskip

We also mention that Theorem~\ref{lsc} 
immediately implies complete delocalization of
each eigenvector of the Wigner matrix: 

\begin{corollary} [Complete  delocalization]
  Let $u_1, u_2, \ldots $ be the
$\ell^2$-normalized eigenvectors of $H$.
Under the assumptions of Theorem~\ref{lsc}  we have
\be\label{deloc}
\P \Bigg\{  \exists \beta \; : \; \|u_\beta\|_\infty^2 \ge \frac{(\log N)^{10 L}}{N} \Bigg\}
 \le  C\exp{\big[-c(\log N)^{\phi L} \big]}
\ee
for any sufficiently large $N\ge N_0$. 
\end{corollary}
For the proof, notice that \eqref{Lambdaodfinal} implies 
 the bound $|G_{jj}(z)| = O(1)$  with very high probability for any $z\in \bS_L$.
Therefore,
$$
   C\ge \im G_{jj}(\la_\al+i\eta)=\sum_\beta
 \frac{\eta |u_\beta(j)|^2}{(\la_\beta-\la_\al)^2+\eta^2} 
\ge \frac{|u_\al(j)|^2}{\eta}.
$$
 The original proof of 
delocalization of eigenvectors was derived from the Stieltjes
transform of the empirical measure \cite{ESY1, ESY3}, motivated by a question posed by T.  Spencer.

\bigskip

{\it Sketch of the  proof  of Theorem~\ref{lsc}.}
For simplicity,
 we will assume here that $E=\re z$ is away from the spectral edges.
The starting point is the following well known formula. 
Let $A$, $B$, $C$ be $n\times n$, $m\times n$ and $m\times m$
matrices and set
\be 
D:=\begin{pmatrix}
    A & B^*  \\
  B& C 
\end{pmatrix}.
\ee
Then for any $1\leq i,j\leq n$, we have 
\be
(D^{-1})_{ij}= \big[(A -B^*C^{-1}B)^{-1}\big]_{ij}.
\label{Dinv}
\ee
Applying this formula to the resolvent matrix $ G=(H-z)^{-1}$,  
we have 
\be
     G_{ii} = 
\frac{1}{ h_{ii} - z- \sum_{k,l\ne i}   h_{ik} G^{(i)}_{kl} h_{li} }  
=\frac{1}{h_{ii}- z - \E_i \sum_{k,l\ne i}  h_{ik} G^{(i)}_{kl} h_{li} - Z_i},
\label{1row}
\ee
where
\be\label{Zdefold}
  Z_i: =  \sum_{k,l\ne i}  h_{ik} G^{(i)}_{kl} h_{li}- 
\sum_{k,l\ne i} \E_i h_{ik} G^{(i)}_{kl} h_{li}.
\ee
Here  $G^{(i)}$ denotes  the resolvent of the $(N-1)\times (N-1)$ minor of $H$
after removing the $i$-th row and column and $\E_i$ denotes the expectation
with respect to the entries in the $i$-th row and column.
Since $G^{(i)}$ is independent of $h_{ik}$ 
and $\E_i h_{ik}h_{li} = \frac{1}{N} \delta_{kl}$, we have
$$
 \sum_{k,l\ne i} \E_i h_{ik} G^{(i)}_{kl} h_{li} = \frac{1}{N}\sum_{k\ne i} G^{(i)}_{kk} =
  \frac{1}{N}\sum_{k} G_{kk} + O\Big(\frac{1}{N}\Big).
$$
Here we used  the interlacing property of eigenvalues between a
 matrix and its minors, which implies that
\be
   \Big| \frac{1}{N} \tr G - \frac{1}{N} \tr G^{(i)}\Big|=
 |m(z)-m^{(i)}(z)|\le \frac{C}{N\eta}, \qquad \eta = \im z>0.
\label{mm}\ee
Defining  $v_i: = G_{ii} - m_{sc}$, we thus have
\be\label{1}
  v_i = G_{ii} - m_{sc}
= \frac{1}{-z- m_{sc}- \Big(\frac{1}{N} \sum_{j} v_j+Z_i -h_{ii} + O( N^{-1})\Big)}
  - m_{sc} .
\ee
Expanding the denominator, using the identity $m_{sc}(z)+ [ m_{sc}(z)+z]^{-1}=0$
 and neglecting the error terms $h_{ii}+ O(N^{-1})= O( N^{-1/2})$, 
we have 
\be
v_i = m^2_{sc} \Big(\frac{1}{N} \sum_{j} v_j + Z_i  \Big)+m^3_{sc} \Big( 
\frac{1}{N} \sum_{j} v_j+  Z_i  \Big)^2+ \ldots 
\ee
Summing  up $i$ and dividing by $N$, we obtain,  modulo negligible errors, 
\be\label{66}
 \barv : = \frac{1}{N} \sum_j v_j 
\approx m^2_{sc}\barv + m_{sc}^3\barv^2+
 m^2_{sc}\barZ   +  O \left (   \frac{1}{N}   \sum_i |Z_i|^2  \right ), \quad 
 \barZ: = \frac{1}{N} \sum_j Z_j.
\ee

To estimate $Z_i$, we 
compute its second moment
\be
 \E |Z_i|^2 = \E\sum_{k,l\ne i}\sum_{k',l'\ne i} \E_i \Bigg( \Big[  h_{ik} G^{(i)}_{kl} h_{li} - 
  \E_i  h_{ik} G^{(i)}_{kl} h_{li}\Big] \Big[  \ov h_{ik'} \ov G^{(i)}_{k'l'} \ov h_{l' i} - 
  \E_i \ov  h_{ik'} \ov G^{(i)}_{k'l'} \ov h_{l' i }\Big]\Bigg).
\label{Zib}
\ee
Since $\E h=0$, the non-zero contributions to this sum come from
index combinations when all $h$ and $\ov h$ are paired.
For pedagogical simplicity, assume that $\E h^2=0$, this can be achieved, for example, if
the distribution of the real and imaginary parts are the same. Then
the $h$ factors in the above
expression have to be paired in such a way that $h_{ik}=h_{ik'}$ and $h_{il}=h_{il'}$,
i.e., $k=k'$, $l=l'$. Note that pairing $h_{ik}=h_{il}$  would give zero because
the expectation is subtracted. The result is
\be
   \E_i |Z_i|^2 = \frac{1}{N^2}\sum_{k,l\ne i} |G^{(i)}_{kl}|^2  + \frac{m_4-1}{N^2} \sum_{k\ne i}
  |G_{kk}^{(i)}|^2,
\label{ZZ}
\ee
where $m_4 =\E |\sqrt{N}h|^4$ is the fourth moment of the single entry
distribution. 
The first term can be computed
\be
    \frac{1}{N^2}\sum_{k,l\ne i} |G^{(i)}_{kl}|^2 =  \frac{1}{N^2}\sum_{k\ne i} 
 (|G^{(i)}|^2)_{kk} 
= \frac{1}{N\eta} \frac{1}{N}
  \sum_k \im G^{(i)}_{kk}  : =   \frac{1}{N\eta} \im m^{(i)}.
\label{egy}
\ee
The second term in \eqref{ZZ}   can be estimated by a similar bound. 
These estimates   confirm that the size of $Z_i$, at least in the second moment sense, is roughly
\be
  |Z_i|\lesssim  \frac{C }{\sqrt{N\eta}}.
\label{Zest}
\ee
Neglecting the $\barv^2$ term in \eqref{66} and using that $|1-m_{sc}^2|\ge c$ 
away from the spectral edge for
 some positive $c$,
we thus  have $ |m(z)-m_{sc}(z)|\lesssim C(N\eta)^{-1/2}$.
A similar but more involved argument gives the same bound
for individual $v_i$'s, showing the estimate \eqref{Lambdaodfinal}
for the diagonal elements $G_{ii}$.
 The estimate 
for the off-diagonal terms, $G_{ij}$, $i\ne j$, is obtained
from the identity $G_{ij}= G_{jj}G_{ii}^{(j)}\big[ Z_{ij}- h_{ij}\big]$
 which can be proved using \eqref{Dinv}.
Here $Z_{ij}$ is defined analogously to \eqref{Zdefold} as
$$
    Z_{ij}: =  \sum_{k,l\ne i,j}  h_{ik} G^{(ij)}_{kl} h_{lj}- 
\sum_{k,l\ne i,j} \E_{ij} h_{ik} G^{(ij)}_{kl} h_{lj},
$$
where $G^{(ij)}$ is the resolvent of the $(N-2)\times (N-2)$ minor of $H$
after removing the $i$-th and $j$-th row and column. The
bound \eqref{Zest} holds for  $Z_{ij}$ as well.

 The estimate for  $\barv = m-m_{sc}$,
the average of $v_i$'s, is of order $(N\eta)^{-1}$
 in \eqref{Lambdafinal}, i.e. it is  better
than the  $(N\eta)^{-1/2}$ estimate for the individual matrix elements
in \eqref{Lambdaodfinal}.
The key mechanism for this improvement
is  the cancellation of the $Z_j$'s 
in their average $\barZ$.   
If $Z_j$'s were independent, we
would gain a factor $N^{-1/2}$ by the central limit theorem.
But $Z_j$'s are correlated and the cancellation takes the following form:

\begin{lemma} [Fluctuation Averaging Lemma] \label{cancel} With the notations of Theorem~\ref{lsc},
for any $\e>0$ we have
\be
\P\left(\frac1N\left|\sum_{i=1}^NZ_i\right|\geq \frac{N^\e}{N\eta }\right)\leq
 C\exp{\big[-c(\log N)^{\phi L} \big]} 
\label{OPU}
\ee
for sufficiently large $N$. 
\end{lemma}

Using this lemma and \eqref{66}, we have proved the
stronger estimate for $\barv$. This completes the sketch of
the proof of the
  local semicircle law,  Theorem \ref{lsc}.\qed

\section{The Green function comparison theorems}\label{sec:4mom}

We now state the Green function comparison theorem, Theorem~\ref{comparison}.
It will  quickly lead to Theorem~\ref{com} stating that 
the correlation functions  of eigenvalues of two matrix ensembles 
are identical  on a scale smaller than $1/N$  provided that the first four moments 
of all  matrix elements  of these two ensembles are almost the same.  We will state 
a limited version  for real Wigner matrices for simplicity of  presentation.

\begin{theorem}[Green function comparison]\label{comparison}\cite[Theorem 2.3]{EYY}  
 Suppose that we have  
two  $N\times N$ Wigner matrices, 
$H^{(v)}$ 
and $H^{(w)}$, with matrix elements $h_{ij}$
given by the random variables $N^{-1/2} v_{ij}$ and 
$N^{-1/2} w_{ij}$, respectively, with $v_{ij}$ and $w_{ij}$ satisfying
the uniform subexponential decay condition \eqref{subexp}.
We assume that the first four moments of
  $v_{ij}$ and $w_{ij}$ are  close to each other in the sense that 
\be\label{4}
    \big | \E  v_{ij}^s  -  \E  w_{ij}^s \big | \le N^{-\delta -2+ s/2},
  \qquad 1\le s\le 4,
\ee
holds for some $\delta > 0$. 
Then  there are positive constants $C_1$ and $\e$,
depending on $\ttau$ and $C_0$ from  \eqref{subexp} such that for any $\eta$ with
$N^{-1-\e}\le \eta\le N^{-1}$  and for any $z_1, z_2$ with $\im z_j = \pm \eta$, $j=1,2$,
we have 
\begin{align}\label{maincomp}
\lim_{N \to \infty} \Big [ \E \tr  G^{(v)}(z_1)  \tr G^{(v)}(z_2)    - \E \tr  G^{(w)}(z_1)  \tr G^{(w)}(z_2)\Big ] 
= 0,
\end{align}
 where $G^{(v)}$ and $G^{(w)}$ denotes the Green functions of $H^{(v)}$ and $H^{(w)}$.
\end{theorem}

The 
matching condition \eqref{4}  is essentially the same as the one appeared in \cite{TV}. 
Here we formulated Theorem \ref{comparison} for a product of two
traces of  the Green function, but  the result 
holds   for  a large class of
smooth functions  depending on several individual
matrix elements of the  Green functions as well,
 see \cite{EYY}  for the precise statement.  (The 
matching condition \eqref{4}  is slightly weaker than in \cite{EYY}, 
but the proof in \cite{EYY} without any change  yields this slightly stronger version.) 
This general version of
 Theorem \ref{comparison} implies the correlation functions 
  of the two ensembles at the scale $1/N$ are identical:

\begin{theorem}[Correlation function comparison]\label{com} \cite[Theorem 6.4]{EYY}
Suppose the assumptions of Theorem \ref{comparison} hold. 
Let $p_{v, N}^{(n)}$ and $p_{w, N}^{(n)}$
be the  $n-$point functions of the eigenvalues w.r.t. the probability law of the matrix $H^{(v)}$
and $H^{(w)}$, respectively. 
Then for any $|E| < 2$,  any
$n\ge 1$ and  any compactly supported continuous test function
$O:\bR^n\to \bR$ we have   
\be \label{6.3}
\lim_{N\to\infty}\int_{\R^n}  \rd\alpha_1 
\ldots \rd\alpha_n \; O(\alpha_1,\ldots,\alpha_n) 
   \Big ( p_{v, N}^{(n)}  - p_{w, N} ^{(n)} \Big )
  \Big (E+\frac{\alpha_1}{N}, 
\ldots, E+\frac{\alpha_n}{N }\Big) =0.
\ee
\end{theorem}
The basic idea for proving Theorem \ref{comparison} is
similar to Lindeberg's proof of the central limit theorem, where
the random variables are replaced one by one with a Gaussian one.
We will replace  the matrix elements $v_{ij}$ with $w_{ij}$ one by one and
 estimate the effect of this change on the resolvent 
by a resolvent expansion. The idea of applying Lindeberg's method in random matrices 
was recently used  by Chatterjee \cite{Ch}  for comparing  the traces of the Green functions; 
the idea was also used by Tao and Vu \cite{TV}
 in the context of comparing individual eigenvalue 
distributions.   There are two main differences between our method and the one 
that appeared in \cite{TV}:
\begin{itemize}
\item[(i)]  We compare the statistics of eigenvalues  of two different ensembles near fixed energies while  
\cite{TV} compared the statistics of the
$j_1, j_2, \ldots j_k$-th  eigenvalues for fixed labels $j_1, j_2, \ldots j_k$.
\item[(ii)]  There is a serious difficulty 
in the approach \cite{TV} concerning 
 possible resonances of neighboring eigenvalues that may render the expansion unstable.
 \end{itemize}
The Green function method eliminates this  difficulty   completely and
Theorem~\ref{comparison} is a simple corollary 
of the Green function estimate Theorem  \ref{lsc}.

\medskip

For a sketch of the proof,
fix a bijective ordering map on the index set of
the independent matrix elements,
\[
\phi: \{(i, j): 1\le i\le  j \le N \} \to \Big\{1, \ldots, \gamma(N)\Big\} , 
\qquad \gamma(N): =\frac{N(N+1)}{2},
\] 
and denote by  $H_\gamma$  the  Wigner matrix whose matrix 
elements $h_{ij}$ follow
the $v$-distribution if $\phi(i,j)\le \gamma$ and they follow the $w$-distribution
otherwise; in particular $H^{(v)}= H_0$ and $H^{(w)}= H_{\gamma(N)}$. 

Consider the telescopic sum of 
differences of expectations (we present only  one resolvent for simplicity of the presentation):
\begin{align}\label{tel}
\E \, \left ( \frac{1}{N}\tr  \frac 1 {H^{(w)}-z} \right )   - 
 & \E \,   \left  (  \frac{1}{N}\tr  \frac  1 {H^{(v)}-z} 
\right )  \\
= & \sum_{\gamma=1}^{\gamma(N)}\left[  \E \, 
 \left (  \frac{1}{N}\tr \frac 1 { H_\gamma-z} \right ) 
-  \E \,  \left (  \frac{1}{N}\tr \frac  1 { H_{\gamma-1}-z} \right ) \right] . \non
\end{align}
Let $E^{(ij)}$ denote the matrix whose matrix elements are zero everywhere except
at the $(i,j)$ position, where it is 1, i.e.,  $E^{(ij)}_{k\ell}=\delta_{ik}\delta_{j\ell}$.
Fix a $\gamma\ge 1$ and let $(i,j)$ be determined by  $\phi (i, j) = \gamma$.
We will compare $H_{\gamma-1}$ with $H_\gamma$.
Note that these two matrices differ only in the $(i,j)$ and $(j,i)$ matrix elements 
and they can be written as
$$
    H_{\gamma-1} = Q + \frac{1}{\sqrt{N}}V, \qquad V:= v_{ij}E^{(ij)}
+ v_{ji}  E^{(ji)}, \qquad v_{ji}:= \ov v_{ij},
$$
$$
    H_\gamma = Q + \frac{1}{\sqrt{N}} W, \qquad W:= w_{ij}E^{(ij)} +
   w_{ji} E^{(ji)}, \qquad w_{ji}:= \ov w_{ij},
$$
with a matrix $Q$ that has zero matrix element at the $(i,j)$ and $(j,i)$ positions.

By the resolvent expansion, 
\[
  S_{\gamma-1}= R - N^{-1/2} RVR + \ldots  + N^{-2} (RV)^4R - N^{-5/2} (RV)^5 S,\quad 
    R := \frac{1}{Q-z}, \; \; S_{\gamma-1} :=  \frac{1}{H_{\gamma-1}-z},
\]
and a similar expression holds for the resolvent $S_\gamma$ of  by $H_{\gamma}$. 
 From the local semicircle law  for individual matrix elements \eqref{Lambdaodfinal}, 
 the matrix elements of all Green functions $R$,  $S_{\gamma-1}, S_\gamma$  are bounded  
by $CN^\e$ for any $\e>0$.
By assumption \eqref{4},  the difference between 
the expectation of matrix elements of  $S_{\gamma-1}$ and $S_\gamma$
 is of order $N^{-2-\delta + C \e}$.
 Since the number of steps, $\gamma(N)$ is of order $N^2$,  the
difference in \eqref{tel} is of order $N^{2}N^{-2 - \delta+ C \e}\ll 1$, and
this proves Theorem \ref{comparison}  for a single resolvent. 
It is very simple to turn this heuristic argument into a rigorous proof
and to generalize it to the product of several resolvents.
 The real difficulty is the input that 
the local semicircle law holds for a general class of Wigner matrices.

\section{Universality for Wigner matrices: putting it together}\label{sec:put}

In this short section we put the previous information together to prove 
Theorem \ref{bulkWigner}. We first focus on the case when $b_N$ is independent of $N$. 
Recall that Theorem~\ref{thm:DBM} states that the correlation 
functions of the Gaussian divisible ensemble,
\be\label{matrixdbm1}
H_t = e^{-t/2} H_0 + (1-e^{-t})^{1/2}\, U,
\ee
where $H_0$  is the initial  Wigner matrix and 
$U$ is an independent standard GUE (or GOE) matrix,
are given by the corresponding GUE (or GOE) for $t\ge  N^{-2\fa+\e}$ 
provided that the a-priori estimate \eqref{assum3}  holds
for the solution $f_t$ of the forward equation \eqref{dy} with some exponent $\fa>0$.
Since the rigidity  of eigenvalues, Corollary  \ref{7.1}, holds uniformly for all Wigner matrices, 
we have proved  \eqref{assum3} for $\fa = 1/2 - \e$  with any $\e>0$.

{F}rom the evolution of the OU process \eqref{zij} for $v_{ij}  = N^{1/2} h_{ij}$ we have 
\be\label{smom}
\big | \E v_{ij}^s(t) -  \E v_{ij}^s(0) \big | \le C t = C N^{-1 + 3 \e}
\ee
for  $s=3, 4$ and with the choice of
 $t = N^{-1 + 3\e}$. 
 Furthermore, $\E h_{ij}^s(t)$ are independent of $t$ for $s=1, 2$
due to $\E v_{ij}(0) = 0$ and $\E v_{ij}^2(t) = 1$. 
Hence \eqref{4} is satisfied  for the matrix elements of $H_t$ and $H_0$ and 
we can thus use Theorem \ref{com} to conclude that the correlation functions of $H_t$ and $H_0$
are identical at the scale $1/N$. 
Since the correlation functions of $H_t$ are given by the corresponding Gaussian 
case, we have proved Theorem \ref{bulkWigner} under the condition
 that the probability distribution of the matrix elements  decay subexponentially. 
Finally, we need a technical cutoff argument to relax the decay condition
which  we omit here (see Section 7 in \cite{EKYY2}).

 The argument for $N$-dependent $b=b_N$ in the range
 $b_N\ge N^{-1+\xi}$, $\xi>0$, is slightly different. For such a small 
$b_N$, \eqref{abstrthm} could be established only for relatively large times, $t\ge N^{-\xi/8}$.
We cannot therefore compare $H_0$ with $H_t$ directly, since the deviation of the third
moments of $v_{ij}(0)$ and $v_{ij}(t)$ in \eqref{smom} would not satisfy
\eqref{4}. Instead, we construct an auxiliary Wigner matrix $\wh H_0$ such that  
up to the third moment {\it its} time evolution $\wh H_t$ under the OU flow \eqref{matrixdbm1}
 matches exactly the {\it original} matrix $H_0$ and the fourth moments are close
 even for $t$ of order $N^{-\xi/8}$
(see Lemma 3.4 of \cite{EYY2}).
Theorem~\ref{thm:DBM} will then be
applied for $\wh H_t$, and Theorem~\ref{comparison} can be used to compare  $\wh H_t$
and $H_0$.

\medskip

We finally discuss  the extension of  Theorem \ref{bulkWigner} without averaging
in $E'$. For {\it Hermitian} matrices, with the notations of  Theorem \ref{bulkWigner},  
for any fixed $|E|<2$ we have  that
\be\label{pointwise}
 \int_{\bR^n} \rd \alpha_1 \cdots
 \rd \alpha_n\, O(\alpha_1, 
\dots, \alpha_n)
 \frac{1}{\varrho_{sc}(E)^n} \left ( {p_{N}^{(n)} - p_{{\rm G}, N}^{(n)}} \right ) 
 \left ( {E +
\frac{\alpha_1}{N\varrho_{sc}(E)}, \dots, E + \frac{\alpha_n}{N\varrho_{sc}(E)}}\right )  \;=\; 0\,.
\ee
This convergence was first proved in Theorem 1.1 of \cite{EPRSY}   for matrices with  distribution which is $ C  n$-times  differentiable for some universal constant  $C$.
 For a general distribution it was stated as Theorem 5 in \cite{TV5}. 
Although the proof in \cite{TV5} took   a  slightly different path, 
this generalization  is an immediate  corollary  of  our previous 
results \cite{EY}. 
Recall our three step approach reviewed in the introduction.  If we substitute 
Step 2b with Step 2a,  then all our  results in the Hermitian case would need no time average. 
More precisely,   Proposition 3.1 of \cite{EPRSY}  asserts that the  bulk universality in the Hermitian
case holds at a fixed energy 
for the Gaussian convolution matrix $H_t$ with $t\sim N^{-1+\delta}$. 
The first four
moments of $H_t$ and $H_0$ are  sufficiently close to apply directly the
Green function
comparison theorem for correlation functions (Theorem~\ref{com} in this article). 
This  concludes the bulk universality
of the original matrix $H_0$ at a fixed energy, which is the  Theorem 5 in \cite{TV5}. 
In fact, our theory implies the same result  for generalized  Hermitian matrices (defined in Section 8) with finite  $4+ \e$ moments.

\section{Beta ensemble: Rigidity  estimates}   \label{beta}

The general $\beta$-ensemble with a potential $V$ is defined by the
probability measure $\mu =\mu_{\beta, V}^{(N)}$  \eqref{01}
 on $N$ ordered real points
$\lambda_1\leq \ldots\leq \lambda_N$. 
We let $\P_\mu$ and $\E_\mu$  denote the probability and the expectation with respect 
to $\mu$.
For simplicity of  presentation we assume that 
 the potential $V$ is convex,  i.e., 
\begin{equation}\label{eqn:LSImu}
\varpi :=\frac{1}{2} \inf_{x\in\RR}V''(x) > 0,
\end{equation}
the equilibrium density  $\varrho(s)$ is supported on a single interval 
$[A,B]\subset \bR$ and satisfies \eqref{equilibrium}
(for the general case, see \cite{BEY2}). 
The  Gaussian case corresponds to 
$V(x)=x^2/2$,
in which case  the equilibrium density is
the semicircle law, $\varrho_{sc}$, given by  \eqref{sc}.
Our  main result concerning the universality is Theorem \ref{bulkbeta} and 
similar  statement holds for the 
universality of the
gap distributions directly. In fact, the proof of Theorem \ref{bulkbeta} goes via 
the gap distribution as we now  explain.

Similarly to \eqref{def:gamma} we again
denote by $\gamma_k$  the classical location 
of the $k$-th point  w.r.t. the limiting  equilibrium
density $\varrho(s)$, i.e. $\gamma_k$ is defined by
\be
 \int_{-\infty}^{\gamma_k}\varrho(s)\rd s=\frac{k}{N}.
\label{classrho}
\ee
The first step to prove Theorem \ref{bulkbeta}  is the following theorem 
which provides a rigidity estimate  on the location of each individual point
in the bulk almost down to the optimal scale $1/N$.
In the following, we will denote $\llbracket x,y\rrbracket=\NN\cap[x,y]$.

\begin{theorem}\label{thm:accuracy}\cite[Theorem 3.1]{BEY}
Fix any $\alpha, \e >0$ and assume that \eqref{eqn:LSImu} holds. Then there are constants
$\delta,c_1,c_2>0$ such that for any $N\geq 1$ and $k\in\llbracket \alpha N,(1-\alpha) N\rrbracket$,
$$
\P_\mu\left(|\lambda_k-\gamma_k|> N^{-1+\e}\right)\leq c_1e^{-c_2N^\delta}.
$$
\end{theorem}

\bigskip

The first ingredient to prove Theorem \ref{thm:accuracy} is an analysis of the loop equation
following Johansson \cite{Joh} and Shcherbina \cite{Sch}. 
 The equilibrium density $\varrho$, for a convex potential $V$,  is given by 
 \begin{equation}\label{eqn:rho}
\varrho(t) =\frac{1}{\pi}r(t)\sqrt{(t-A)(B-t)}\mathds{1}_{[A,B]}(t),
\end{equation}
where $r$ is a real function  that
can be extended to an analytic function in $\CC$  and $r$ 
has no zero in $\RR$.
Denote by  $s(z):=-2r(z)\sqrt{(A-z)(B-z)}$ where the square root is defined such
 that its asymptotic value 
is $z$ as $z\to\infty$. Recall that the density
is the one-point correlation function which is characterized  by
\be
\int_\bR \rd \lambda_1   O(\lambda_1)  p^{(1)}_N(\lambda_1) =
\int_{\bR^N}  O(\lambda_1) 
 \rd\mu_{\beta, V}^{(N)}(\lambda), \qquad \lambda = (\la_1, \la_2, \ldots, \la_N).
\label{def:dens}
\ee
  Let  $\bar m_N$ and $m$ be the Stieltjes transforms of the  density $p_N^{(1)}$
and the equilibrium density $\varrho$, respectively.  Notice that
in Section~\ref{sec:refined}  we have used $m=m_N$ to denote 
the Stieltjes transform of the  empirical measure \eqref{St}; 
here $\bar m_N$ denotes the ensemble average of the analogous quantity.

Define the analytic functions 
$$
b_N(z):=\int_{\RR}\frac{V'(z)-V'(t)}{z-t}(p_1^{(N)}-\varrho)(t)\ind t 
$$
and  $c_N(z):=\frac{1}{N^2}k_N(z)+\frac{1}{N}\left(\frac{2}{\beta}-1\right)\bar m_N'(z)$,
where $k_N(z):=\var_\mu\left(\sum_{k=1}^N\frac{1}{z-\lambda_k}\right).$
Here for  complex random variables $X$ we use the definition that 
$\var(X)=\E(X^2)-\E(X)^2$.

 The equation used by Johansson 
(which can be obtained by a change of variables
in \eqref{def:dens}
 \cite{Joh} or by integration by parts \cite{Sch}), is
 a variation of the loop equation (see, e.g., \cite{Ey})
used in the physics literature and it takes the form
\begin{equation}\label{eqn:firstLoop}
( \bar m_N-m)^2+s(\bar  m_N-m)+b_N=c_N.
\end{equation}

Equation  \eqref{eqn:firstLoop} expresses the difference $\bar m_N - m$ 
in terms of $(\bar m_N-m)^2$, $b_N$ and $c_N$.
In the  regime where $|\bar m_N - m|$  is small,  we 
can neglect the quadratic term.  The term $b_N$ is of the same order as
$|\bar m_N-m|$ and is  difficult  to treat. As observed in \cite{APS,Sch}, for analytic  $V$,
  this term vanishes
when we perform a contour integration. So we have roughly the relation
\be\label{55}
( \bar m_N-m) \sim \frac 1 { N^2} \var_\mu\left(\sum_{k=1}^N\frac{1}{z-\lambda_k}\right),
\ee
where we dropped the less important error involving $\bar m_N'(z)/N $ due to the extra $1/N$ factor.
 In the convex setting,
the variance can be estimated by  the logarithmic Sobolev
inequality and we immediately obtain an estimate on $\bar m_N - m$.
We then use the Helffer-Sj\"ostrand formula, see \eqref{1*}, 
 to  estimate the locations of the particles. 
 This will provide us with an accuracy of order $N^{-1/2}$
for $\E_\mu \lambda_k-\gamma_k$.  This argument gives only
an estimate on the expectation of the locations of the particles  since 
we only have  information
 on the  averaged quantity, $\bar m_N$. 
 Although it is tempting
to use  this new accuracy information on the particles
 to estimate the variance again in \eqref{55}, the information on 
 the expectation on $\lambda_k$ alone is very difficult to use 
in a bootstrap argument.
To estimate the variance of a non-trivial function of $\lambda_k$
 we need high probability estimates 
on  $\lambda_k$.

The key idea in this section is the observation that the 
accuracy information on the $\lambda$'s can be used to  improve the local convexity
of the measure $\mu$ in  the direction involving the {\it differences} of $\lambda$'s.
To explain this idea, we compute the  Hessian of the Hamiltonian of $\mu$:
$$
\Big\langle \bv , \nabla^2  \cH(\la)\bv\Big\rangle
\ge   \varpi  \,   \|\bv\|^2 + \frac{1}{N}
 \sum_{i<j} \frac{(v_i - v_j)^2}{(\la_i-\la_j)^2}. 
$$
The naive lower bound on $\nabla^2  \cH$ is $\varpi$, but for a
 typical $\lambda=(\la_1, \la_2, \ldots, \la_N)$ it is in fact much better
in most directions. To see this effect,  suppose  
we know $ |\la_i-\la_j| \lesssim M /N$ 
 with some $M$ for any  $i, j  \in I_k^M$, where
$I_k^M := \llbracket k-M, k+M\rrbracket$. 
Then for $\bv = (v_{k-M}, \ldots, v_{k+M})$
with $\sum_j v_j = 0$ we have 
\be\label{convex2}
\Big\langle \bv , \nabla^2  \cH(\la)\bv\Big\rangle
\ge \frac{N}{M^2}
 \sum_{i, j \in I_k^M}  (v_i - v_j)^2 \ge C \frac N M \sum_j v_j^2.  
\ee
 This improves the convexity of the Hessian to $N/M$  on the hyperplane $\sum_j v_j = 0$.
Let 
$$
\lambda_k^{[M]}: = |I_k^M|^{-1}\sum_{j\in I_k^M}\lambda_j
$$
denote the block average 
of the locations of particles and rewrite 
$$ 
 \lambda_k - \lambda_k^{[N^{1-\e}]}=  \sum_j \Big( \lambda_k^{[M_j]}- \lambda_k^{[M_{j+1}]}\Big)
$$
as a telescopic sum with an appropriate sequence of $M_1=0$, $M_2, \ldots$.
We can now use the improved concentration on the hyperplane $\sum_j v_j = 0$ to the 
variables $\lambda_k^{[M_j]}- \lambda_k^{[M_{j+1}]}$ to
 control the fluctuation of $ \lambda_k - \lambda_k^{[N^{1-\e}]}$. 
Since the fluctuation of $  \lambda_k^{[N^{1-\e}]}$ is very small for small $\e$, 
we finally arrive at the estimate 
\be\label{concen1}
\P_\mu \left(|\la_k-\E_\mu(\la_k)|> a \right)\leq C e^{- C  N^2 a^2/ M }.
\ee

{F}rom \eqref{concen1} we thus  have that $|\la_k - \E_\mu \la_k| \lesssim \sqrt M /N$
with high probability. This improves the starting accuracy 
$ |\la_i-\la_j| \lesssim M /N$  for $i, j  \in I_k^M$  to $ |\la_i-\la_j| \lesssim M' /N$
with some $M'\ll M$,
 provided that we can prove  that  $ |\E_\mu (\la_i-\la_j )| \ll  M' /N$.
 But the last inequality 
involves only expectations and it
will follow from the analysis of the loop equation \eqref{eqn:firstLoop} 
we just mentioned  above.  Starting from $M=N$, this procedure can
 be repeated by decreasing $M$ step by step until we get the optimal 
accuracy, $M \sim O(1)$.
The implementation of this argument in \cite{BEY} is somewhat
different from this sketch due to various technical issues, but it follows the same basic idea.

\section{Beta ensemble: The local equilibrium measure}\label{sec:loceq}

\medskip 

Having completed the first step, the rigidity
  estimate, we now focus on the second
step, i.e. on the uniqueness of the local Gibbs measure.
Let $0<\kappa<1/2$. Choose $q\in [\kappa, 1-\kappa]$ and set $L=[Nq]$ (the integer part).
 Fix an integer  $K=N^k$ with $k<1$.
 We will study the local spacing statistics
of $K$ consecutive particles
$$
  \{ \lambda_j\; : \; j\in I\}, \qquad I=I_L:=
 \llbracket L+1, L+K \rrbracket.
$$
These particles are typically located near
$E_q$ determined by the relation
$$
  \int_{-\infty}^{E_q} \varrho(t) \rd t = q.
$$
Note that $|\gamma_L- E_q|\le C/N$.

We will distinguish the inside and outside particles
by renaming them as
\be\label{35}
(\lambda_1, \lambda_2, \ldots,
\lambda_N):=(y_{1}, \ldots y_{L}, x_{L+1},  \ldots, x_{L+K}, y_{L+K+1},
  \ldots y_{N}) \in \Xi^{(N)},
\ee
but note that they keep their original indices.
The notation $\Xi^{(N)}$ refers to the simplex
$\{\bz \; :\; z_1<z_2< \ldots < z_N\}$ in $\RR^N$.
In short we will write
$$
\bx=( x_{L+1},  \ldots, x_{L+K} ), \qquad \mbox{and}\qquad
 \by=
 (y_{1}, \ldots, y_{L}, y_{L+K+1},
  \ldots, y_{N}),
$$
all  in increasing order, i.e. $\bx\in \Xi^{(K)}$ and
$\by \in \Xi^{(N-K)}$.
We will refer to the $y$'s as   {\it external
points}  and to the $x$'s as  {\it internal points}.

We will fix the external points (also called
as boundary conditions) and study
conditional measures on the internal points.
We  define  the
{\it local equilibrium measure} on $\bx$ with fixed boundary condition  $\by$ by
\begin{equation}\label{eq:muyde}
 \quad
\mu_{\by} (\rd\bx)  = \mu_\by(\bx) \rd \bx, \qquad
\mu_\by(\bx):=  \mu (\by, \bx) \left [ \int \mu (\by, \bx) \rd \bx \right ]^{-1}.
\end{equation}
Note that for any fixed $\by\in \Xi^{(N-K)}$,  
the measure $\mu_\by$ is supported on configurations of $K$ points 
$\bx=\{ x_j\}_{j\in I}$ 
 located in the interval $[y_{L}, y_{L+K+1}]$.

The Hamiltonian $\cH_\by$ of the measure $\mu_\by (\rd \bx) \sim \exp(-N\cH_\by(\bx))\rd\bx$
 is given by 
\be\label{24}
\cH_{\by} (\bx) :=
 \sum_{i\in I}  \frac{\beta}{2}V_\by (x_i)
-  \frac{ \beta }{N} \sum_{i,j\in I\atop i< j}
\log |x_{j} - x_{i}| 
\qquad \mbox{with}\qquad
V_\by (x): = V(x) - \frac{ 1 }{ N} \sum_{j \not \in I}
\log |x - y_{j}|.
\ee
We now define the set of {\it good boundary configurations}
with  a parameter $\delta=\delta(N)>0$
\begin{align}\label{goodset}
  \cG_{\delta}=\cG:=
\Big\{   \by \in \Xi^{(N-K)}\; :\; |y_j-\gamma_j|\le \delta, 
\; \forall \, j\in
\llbracket N\kappa/2, L\rrbracket \cup \llbracket L+K+1,
N(1-\kappa/2)\rrbracket
\Big\},
\end{align}
where $\kappa$ is a small constant to cutoff points near the spectral edges.  Some 
rather weak  additional conditions
 for $\by$ near the spectral edges will also be needed, but  we will neglect
this issue here.

Let   $\sigma $ and $ \mu$ be two measures of the form \eqref{01}  
with  potentials $W$ and $V$ and 
 densities $\varrho=\varrho_W$ and $\varrho_V$, respectively.
For our purpose $W(x)=x^2/2$,  i.e., $\sigma$ 
is the Gaussian $\beta$-ensemble
and $\varrho_W(t) =\frac{1}{2\pi}(4-t^2)^{1/2}_+$ is the Wigner semicircle law.
Let  the sequence $\gamma_j$ be the classical locations for $\mu$ and
the sequence $\theta_j$ be the classical locations for $\sigma$.
Similarly to the construction of the  measure $\mu_{\by}$,
 for any positive integer $L' \in \llbracket 1, N-K \rrbracket$  we can construct 
the measure $\sigma_\bt$ conditioned that the particles outside are given by 
the classical locations $\theta_j$ 
for $j \notin
\llbracket  L',  L'+K \rrbracket$. 
More precisely,  we define a {\it reference
local Gaussian  measure} $\sigma_\th$ on the set $[\th_{L'}, \th_{L'+K+1}]$ via the  Hamiltonian
\be\label{241}
\cH_{\th} (\bx) =
 \sum_{i \in I'}  \Big [  \frac{\beta}{4}   x_i^2   - \frac{ \beta }{N} \sum_{j \not \in I'}
\log |x_i - \th_{j} | \Big ]
-  \frac{ \beta }{N} \sum_{ i,j \in I'\atop i<j }
\log |x_{j} - x_{i}| ,
\ee
where  $I':= \llbracket L'+1, L'+K\rrbracket$.
Since $L'$ will not play an active role, we will 
abuse the notation and set $L'=L$. 

The measure $\mu_{\by}$ lives on the interval 
$[y_L, y_{L+K+1}]$ while the measure $\sigma_{\bt}$ lives on the interval 
$[\theta_L, \theta_{L+K+1}]$ and it is difficult to compare them. 
But after an appropriate translation and dilation, they will live on the same interval and from now on 
we assume that  $[y_L, y_{L+K+1}]= [\theta_L, \theta_{L+K+1}]$. 
The parameter $K=N^k$ has to be sufficiently small since  $\varrho_V$ and 
$\varrho_W$ are not constant functions
and we have 
to match these two densities quite precisely  in the whole interval. 
There are some other subtle  issues related to the rescaling, 
but  we will neglect them here to concentrate on the main ideas.  
Our main result is the following theorem
which is essentially a combination of Proposition 4.2 and Theorem 4.4 from  \cite{BEY}.

\begin{theorem}\label{thm:mi} 
Let $0<\varphi\le \frac{1}{38}$.
 Fix $K=N^k$, $\delta = N^{-d}$
with $d=1-\varphi$ and $k=\frac{39}{2}\varphi$.
Then for $\by\in \cG$  we have 
\be\label{381}
\Bigg|  \E_{ \mu_{\by} } \frac 1 K \sum_{i \in I} G\Big( N(x_i-x_{i+1}) \Big)
   -\E_{ \sigma_{\th} }
 \frac 1 K \sum_{i \in I} G\Big( N(x_i-x_{i+1}) \Big)\Bigg|
\to 0
\ee
as $N \to \infty$
for any smooth and compactly supported test function $G$.
A similar formula holds for more complicated observables of the form
\eqref{cG}.
\end{theorem}

The basic idea for proving Theorem \ref{thm:mi}  is to use the 
 Dirichlet  form
 inequality \eqref{diff}.
Although  \eqref{diff}  was stated for an infinite volume measure, it holds for any measure 
with  repulsive logarithmic interactions in a finite volume 
and with  the parameter $\tau^{-1}$ being 
the lower bound on the Hessian of the Hamiltonian.
In our setting, 
we denote by $\tau_\sigma^{-1}$  the  lower bound for $\nabla^2\cH_\sigma$,
 and the Dirichlet form inequality becomes 
\be\label{relax}
   \Bigg| \big [\E_{ \mu_{\by}} -\E_{ \sigma_{\bt}}  \big  ]
\frac 1 K \sum_{i \in I} G\Big( N(x_i-x_{i+1}) \Big)\Bigg|
\le C \Big(  \frac { \tau_\sigma  N^{\e}}{ K}
 D \big  (\mu_{\by} | \sigma_{\th}   \big  ) \,    \Big)^{1/2}
+ Ce^{-cN^{\e}} \sqrt{S (\mu_{\by} | \sigma_{\bt}   \big  )},
\ee
where 
\be
  D(\mu_\by\mid \sigma_\bt) : =
 \frac{1}{2 N} \int \Big|\nabla \sqrt{ \frac{\rd\mu_\by}{\rd\sigma_\bt}}\Big|^2\rd \sigma_\bt.
\ee
Thus our task is to prove  that 
\be\label{d1}
 \tau_\sigma  N^{\e} \frac {   D(\mu_\by\mid \sigma_\bt)  }{K}  \to 0 .
\ee

By definition,
\[  \frac{\tau_\sigma}{K}  D \big(\mu_\by \mid \sigma_\th)
 \le
\frac{\tau_\sigma N}{ K} \int  \sum_{ L+1\le  j \le L+ K }  Z_j^2 \rd  \mu_{\by},
\]
where $Z_j$ is defined as
\begin{align}
Z_j := &
 \frac{\beta}{2}V'(x_j) -
  \frac \beta N \sum_{    k < L\atop k >L+K }
   \frac 1 {x_j- y_{k}}
 -  \frac{\beta}{2}W'(x_j)  +
 \frac \beta N \sum_{  k <  L\atop k>L+K}   \frac 1 {x_j-  \th_{k}}.
 \label{Zdef}
\end{align}
Using  the equilibrium relation \eqref{equilibrium}
between the potentials $V$, $W$ and the densities $\varrho_V$, $\varrho_W$,
 we have 
\begin{align}
Z_j = &
 \beta\int_\bR  \frac{\varrho_V(y)}{x_j-y}\rd y -
  \frac \beta N \sum_{    k < L\atop k >L+K }
   \frac 1 {x_j- y_{k}}
 - \beta\int_\bR  \frac{\varrho_W(y)}{x_j-y}   \rd y +
 \frac \beta N \sum_{  k <  L\atop k>L+K}   \frac 1 {x_j-  \th_{k}}.
\nonumber
\end{align}
Hence  $Z_j$ is the sum of the error terms,
\begin{align}
A_j  : &= \int_{ y\not \in [y_{L}, y_{L+K+1} ] }    \frac {\varrho_V(y) } {x_j- y }  \rd y
- \frac{1}{N} \sum_{    k < L\atop k >L+K }
   \frac 1 {x_j- y_{k}} ,
\label{Adef}\\
B_j & : = \int_{y_{L}}^{ y_{L+K+1} }    \frac {\varrho_V(y) - \varrho_{W}(y)  } {x_j- y } \rd y,
\label{fiveom}
\end{align}
and there is a term  similar to $A_j$ with $y_j$ replaced by $\theta_j$ and $\varrho_V$
 replaced by $\varrho_W$.

 With our convention,  the total numbers of particles 
in the interval $[y_{L+K+1}, {y_{L}}]$ are equal and thus 
$$
\int_{y_{L}}^{ y_{L+K+1} }  \varrho_V(y)   \rd y = \int_{y_{L}}^{ y_{L+K+1} }  \varrho_{W}(y)  \rd y.
$$
Since the  densities  $\rho_V$ and $\rho_W$ are $C^1$ functions away from the endpoints  $A$ and $B$
and $  y_{L+K+1}- {y_{L}}$ is small, $ |\rho_V-\rho_W|$ is small in the interval $[y_{L+K+1}, {y_{L}}]$ and
thus   $B_j$ is small.
For estimating $A_j$, we can replace the integral $\int_{ -\infty}^{y_{L}}    
 \frac {\varrho_V(y) } {x_j- y }  \rd y$ by 
$\frac 1 N  \sum_{  k< L   }  \frac 1 {x_j- \gamma_{k}}$ with negligible errors,
 at least for $j$'s away from the edges, 
$j\in \llbracket L+N^\e, L+K-N^\e\rrbracket$.
 Thus
\be\label{b7}
|A_j |  \le   \frac C N  \Big | \sum_ { k <  L\atop k>L+K} T^k_j
  \Big|,  \qquad T^k_j :=  \frac 1 {x_j- y_{k}} -  \frac 1 {x_j- \gamma_{k}} ,
\ee
and $T^k_j$ can be estimated by the assumption
  $|y_k -\gamma_k |\le \delta$ from $\by\in \cG$. The same argument works if $j$ is 
close to the edge, but $k$ is away  from the edges, i.e. $k\le L-N^\e$ or $k\ge L+K+N^\e$.
The edge terms,  $ T^k_j$ for $|j-k|\le N^\e$,  are difficult to estimate 
due to the singularity in the denominator  and  the event that  many $y_k$'s
with $k<L$ 
may pile up near $y_L$. 
 To resolve this difficulty, 
 we show that   the averaged local statistics 
of the measure $\mu_\by$  
 are insensitive to the change of the 
boundary conditions for $\by$ near the edges.  This can be achieved  by 
the simple inequality 
\be\label{Diff1}
\Big|\frac 1 K \sum_{i \in I}   \int G\big( N(x_i-x_{i+1} )\big)   [ \rd \mu_{\by'} - \rd  \mu_{\by}]
\Big|
\le C   \int |\rd \mu_{\by'} - \rd  \mu_{\by} |  \le C  \sqrt {   S(\mu_{\by'}|   \mu_{\by} ) }
\ee
for any two boundary conditions
$\by$ and $\by'$.   Although we still have to estimate the entropy
that includes a logarithmic singularity, this can be done much more easily.
Therefore, we can replace the boundary condition $y_k$ with $y_k'=\theta_k$ for 
$|j-k|\le N^\e$  and  then the most singular 
edge terms in \eqref{Zdef} cancel out.

We note that  we can 
perform this replacement only 
for a small number of index pairs $(j,k)$,
 since
estimating the gap distribution by 
the  total entropy, as noted in \eqref{entbound} in Section~\ref{sec:DBM},
is not as efficient as the estimate using  the Dirichlet form per particle. 
Thus we can afford to use this argument only for the edge terms, $|j-k|\le N^\e$.
For all other index pairs $(j,k)$ we still have
to estimate $T_j^k$ by exploiting that $\by$  is a good configuration,
i.e. $y_k-\gamma_k$ is small.

 Unfortunately,   even with the optimal 
accuracy $\delta\sim N^{-1 + \e'}$ in \eqref{goodset} as an input, 
the relation \eqref{d1} 
still cannot be satisfied  for any choice of $N^{c \e'}
 \le K \le N^{ 1 - c\e'}$. 
We do not know whether this is due to our handling of
 the edge terms or some other  intrinsic reasons. 
To understand why this might occur, we remark that while  the edge
terms become a smaller  percentage 
of the total  terms in \eqref{Diff1} as $K$ gets bigger, 
the relaxation time to equilibrium for $\sigma_\bt$,  determined by the convexity  of
$\cH_\bt''$,   increases at the same time. 
  At the end of our calculation, 
there is no good regime for the choice 
of $K$.  Fortunately, this can be resolved by  using the idea of
the local relaxation   measure \cite{ESYY}, i.e., we add 
a quadratic term $\frac{1}{2 \tau } (x_j -\gamma_j)^2$ to the measure
 $\mu_\by$ and $\frac{1}{2 \tau_\sigma } (x_j -\th_j)^2$
to the measure $\sigma_\bt$.  With these ideas, we can complete the proof of Theorem \ref{thm:mi}.

\section{More general classes of random  matrices}\label{sec:gen}

All our results concerning Wigner matrices hold for 
 a broader class of ensembles where
the matrix elements $h_{ij}$ still  have  mean zero, $\E \, h_{ij}=0$, but  
 their   variances  are allowed to vary. More precisely, 
we assume that the variances $ \sigma_{ij}^2: = \E |h_{ij}|^2$
satisfy the normalization condition
\be
  \sum_{j=1}^N \sigma_{ij}^2 =1, \qquad i=1,2,\ldots, N,
\label{normal}
\ee
and they are comparable, i.e. 
\be
 0< C_{inf}\le   N\sigma_{ij}^2  \le C_{sup} < \infty, \qquad i,j=1,2, \ldots, N,
\label{VV}
\ee
for some fixed positive constants $C_{inf}$ and $C_{sup}$. These ensembles are called 
 {\it generalized Wigner ensembles}. 
In the special case  $\sigma_{ij}^2=1/N$, we recover the
original   Wigner ensemble. All our results concerning the bulk universality, delocalization 
of eigenvectors 
and local semicircle laws hold for  generalized Wigner matrices as well.

There is another important class of random matrices, 
 the band matrices, which are characterized by the property that $\sigma_{ij}^2$ is
a function of $|i-j|$ on scale $W$, which is called the bandwidth, i.e.,  
\be\label{BM}
   \sigma^2_{ij} = W^{-1} f\Big(\frac{ [i-j]_N}{W}\Big),
\ee
where  $f:\bR\to \bR_+$ is a bounded nonnegative symmetric function with 
$\int f =1$ 
and $[i-j]_N\equiv i-j \; \mbox{mod}\,\,\, N$. 
 For this class, the local semicircle law is known to hold 
at least down to  scale $\eta \sim W^{-1}$ and
all eigenvectors are delocalized at least on scale $W$.
Moreover, most eigenvectors are known to be delocalized on a
much larger scale $W^{7/6}$ \cite{EK, EK2},  but smaller than $W^8$
\cite{Sche}, and it is
expected  that the correct localization length is $W^2$.
So far no bulk universality result is known.

 The significance of the  random band matrices stems from the fact
that they interpolate between discrete random Schr\"odinger operators
with short range hoppings (Anderson model) and  the
Wigner matrices.  In particular, random matrix spectral statistics are
expected to hold in the presumed delocalization regime of the Anderson model
in three or higher dimensions.  For more details on this exciting connection,
see \cite{Spe}.

Finally we mention the ensemble of sample covariance matrices that
play a fundamental role in statistics. These are matrices of the form $H=A^*A$
where $A$ is an $M\times N$ matrix with independent identically distributed entries.
The semicircle law is replaced with the Marchenko-Pastur law, but   most  results
listed in this review remain valid. For more details, see \cite{ESYY, TV3}.

\section{Edge universality}\label{sec:edge}

 Denote by  $\lambda_N$ is the largest eigenvalue of a generalized   Wigner  matrix.
The 
probability distribution  functions of $\lambda_N$ for the classical Gaussian ensembles are
identified by Tracy and Widom  \cite{TW, TW2} to be  
\be\label{Fb}
\lim_{N \to \infty} \P( N^{2/3} ( \lambda_N -2) \le s ) =  F_\beta (s), 
\ee
where the functions $F_\beta(s)$ can be computed in terms 
of Painlev\'e equations and $\beta=1, 2, 4$
 corresponds to the standard classical  ensembles. The distribution of $\lambda_N$ 
is believed to be universal and independent of the Gaussian structure.

The local semicircle law, Theorem \ref{lsc},
combined with a modification of the Green function comparison theorem,
Theorem \ref{comparison},
implies  the following 
version of universality of the extreme eigenvalues. Although it holds for correlation 
functions of finite number of eigenvalues, for simplicity
 we state it for the largest one 
 and for the case of symmetric matrices only.

\begin{theorem}[Universality of the largest eigenvalue]\cite[Theorem 2.4]{EYYrigi} \label{twthm}  
Suppose that we have 
two $N\times N$  symmetric generalized Wigner matrices,
 $H^{(v)}$ and $H^{(w)}$, with matrix elements $h_{ij}$
given by the random variables $N^{-1/2} v_{ij}$ and 
$N^{-1/2} w_{ij}$, respectively, with $v_{ij}$ and $w_{ij}$ satisfying
the uniform subexponential decay condition \eqref{subexp}. Let $\P^\bv$ and
$\P^\bw$ denote the probability and let $\E^\bv$ and $\E^\bw$  denote
the expectation with respect to these collections of random variables.
Suppose that 
\be\label{2m}
    \E^\bv v_{ij}^2  =  \E^\bw w_{ij}^2. 
\ee
Then there is an $\e>0$ 
depending on $\ttau$ in \eqref{subexp}  such that 
for any real parameter $s$ (may depend on $N$)   
we have
 \be\label{tw}
 \P^\bv ( N^{2/3} ( \lambda_N -2) \le s- N^{-\e} )- N^{-\e}  
  \le   \P^\bw ( N^{2/3} ( \lambda_N -2) \le s )   \le 
 \P^\bv ( N^{2/3} ( \lambda_N -2) \le s+ N^{-\e} )+ N^{-\e}  
\ee 
for $N\ge N_0$ sufficiently  large, where $N_0$ is independent of $s$. 
\end{theorem}
Note that although Theorem~\ref{twthm} states that the edge
distribution is universal for a fixed choice of the variances $\sigma_{ij}^2$,
it does not identify this distribution.
In particular, we do not know if it coincides with the Tracy-Widom
 distribution apart from the Hermitian case, when
the method of \cite{Joh} can be applied.
The extension of Theorem~\ref{twthm} to eigenvectors was recently obtained by 
Knowles and Yin \cite{KY}, i.e.,  under the assumption \eqref{2m}, the  
distributions for the largest eigenvectors coincide.
Similar results hold for the joint  distribution of eigenvectors near the edges. 

\section{Erd\H{o}s-R\'enyi matrix}\label{sec:ER}

The  Erd\H{o}s-R\'enyi matrix is  the
adjacency matrix of the Erd\H{o}s-R\'enyi random graph \cite{ER1, ER2}.
 Its entries are independent (up to the constraint that 
the matrix be symmetric) and are equal to $1$ with probability $p$ and $0$ with probability $1 - p$. 
We  rescale the matrix in such a way that its bulk eigenvalues typically
 lie in an interval of size of order one.
Thus we have a symmetric $N\times N$ matrix 
 $A = (a_{ij})$  whose entries 
$a_{ij}$ are independent
(up to the symmetry constraint $a_{ij}=a_{ji}$) and each element is distributed according to
\begin{equation}\label{sparsedef}
a_{ij} \;=\;
\frac{\gamma}{q}
\begin{cases}
1 & \text{with probability } \frac{q^2}{N}
\\
0 & \text{with probability } 1 - \frac{q^2}{N}\,, 
\end{cases} \quad q := \sqrt { pN }.
\end{equation}
Here $\gamma \deq (1 - q^2 / N)^{-1/2}$ is a scaling 
introduced for convenience to compare with Wigner matrices.  
We also assume that $q= \sqrt{pN}\ge (\log N)^{C\log\log N}$, in 
particular the Erd\H{o}s-R\'enyi  graph is connected.

\begin{theorem}[Local semicircle law for Erd\H{o}s-R\'enyi
 matrix]\cite[Theorem 2.9]{EKYY1}\label{lscer}{}
Let $m(z)$ denote the  Stieltjes transform of the empirical 
eigenvalue distribution of the matrix $A$ and let $G(z)=(A-z)^{-1}$ be its resolvent.
Assume that the spectral parameter $z=E+i\eta$ satisfies $|E|\le 5$ and 
$(\log N)^L N^{-1}\le\eta \le 3$ with a sufficiently large constant $L$. 
Then we have the following two estimates:

(i) The Stieltjes transform of the empirical 
eigenvalue distribution of  $A $  satisfies 
\be\label{LambdafinalER} 
\P \Big\{ |m(z)-m_{sc}(z)| 
 \ge (\log N)^{C}  \Big [ \frac{1}{N\eta} + \frac 1 q  \Big ] \Big\}  
\le  C\exp{\big[-c(\log N)^{c} \big]}. 
\ee

(ii) The individual  matrix elements of
the Green function  satisfy that 
\be\label{LambdaodfinalER}
\P  \left\{ \max_{i,j}|G_{ij}(z)-\delta_{ij}m_{sc}(z)| \geq 
(\log N)^{C} \Big [ \sqrt{\frac{\im m_{sc}(z)  }{N\eta}} + \frac{1}{N\eta} + \frac 1 q
 \Big]  \right\}  
\leq  C\exp{\big[-c(\log N)^{c} \big]}.
\ee
\end{theorem}

Compared with the  local semicircle law, 
Theorem \ref{lsc}, there is an extra factor $1/q$ appearing in the error estimates
of Theorem \ref{lscer}.
This extra  error term affects  
 the rigidity estimate of eigenvalues, and \eqref{7.1} becomes 
\begin{equation} \label{rigidity for large phi}
\abs{\mu_j - \gamma_j} \;\leq\; (\log N)^{C}   \Big [ N^{-2/3}   j^{-1/3} + q^{-2} \Big ] \,,
\qquad  j\le N/2,
\end{equation}
for $q = \sqrt {p N}  \gg N^{1/3}$.  We also have an estimate 
for the regime $q \le N^{1/3}$ but that is weaker.
Moreover,
  under the assumption  
$q \gg  N^{1/3}$,  both bulk and edge universality are proved
(see Theorem 2.5 and 2.7 in \cite{EKYY1}). 
It is well-known that the largest eigenvalue $\la_N$ of $A$ 
satisfies
\begin{equation} \label{position of mu max}
\la_N \;=\; \gamma q + \frac{1}{\gamma q} + o(1)\,,
\end{equation}
hence  it is located far away from the bulk spectrum.
Therefore the edge universality for  Erd\H{o}s-R\'enyi matrices
refers to the {\it second} largest eigenvalue instead of the largest one. 
Since the matrix elements of $A$ have  nonzero means, both the edge and bulk  universality require 
substantial new ideas in addition to those we have sketched. 
  We refer the interested readers to the original papers \cite{EKYY1, EKYY2}
for more detailed explanations.

\section{Historical Remarks}\label{sec:hist}

Finally we summarize the recent  history  related to the  universality 
of local eigenvalue statistics of Wigner matrices. 
The three-step approach was first introduced in \cite{EPRSY} in the context
 of Hermitian Wigner matrices and it  
 led to the first proof of the  Wigner-Dyson-Gaudin-Mehta  conjecture for Hermitian Wigner matrices.
 It works 
whenever the   distributions of the matrix elements are  smooth.   This approach 
was followed by all later works on the bulk universalities.  
We now review the history of Steps 1-3 
separately 
 and we  start with the history of Step 1, the local semicircle law.

The semicircle law was proved by Wigner for energy windows of order  one. 
Various improvements were made to shrink  the spectral windows;   in particular, 
results down  to scale $N^{-1/2}$ were obtained by \cite{BMT} and \cite{GZ}. 
The result at the  optimal scale, $N^{-1}$,  referred to as 
the local semicircle law,
 was established for Wigner
 matrices in a series of papers  \cite{ESY1, ESY2, ESY3}. 
The  method  was based on a self-consistent equation for 
the Stieltjes  transform of the eigenvalues, $m(z)$,  and the continuity in the 
imaginary part of the spectral parameter $z$.  As a by-product, the optimal
eigenvector delocalization estimate was  proved. 
In order to deal with the  generalized Wigner matrices, we needed to consider 
the  self-consistent equation of $G_{ij}(z)$,  the 
matrix elements of the Green function,   since there is no closed equation for $m(z) = N^{-1} \tr G(z) $ 
\cite{EYY, EYY2}.  In particular,   this method implied the optimal rigidity estimate 
of eigenvalues in the bulk  in \cite{EYY2} and 
up to   the edges in \cite{EYYrigi}. The  
estimate on $G_{ii}$ provided a simple alternative proof of the eigenvector delocalization estimate. 
The extension of the local semicircle law to the Erd\H os-R\'enyi matrices was recently made  in 
\cite{EKYY1}.

 We now review the history of Step 2. Recall that
{\it Hermitian}   Gaussian divisible ensembles  are  matrices of the form 
$e^{-t/2}H_0+ (1-e^{-t})^{1/2} U$,   where $U$ is the GUE and $H_0$ is a 
Wigner ensemble. 
The universality of this ensemble  for  a  large 
class of $H_0$ and  for parameters $t$ of order one   was  proved by Johansson \cite{J}.
 It was extended to {\it complex} sample 
covariance matrices by 
Ben Arous and  P\'ech\'e \cite{BP}. 
There were two major restrictions of this
 method: 1. The Gaussian component 
was fairly large, it was required  to be of order one independent of $N$; 
2. The method relies on an explicit formula by Br\'ezin-Hikami \cite{BH}
 for the correlation functions of 
eigenvalues.  This formula originates in the Harish-Chandra-Itzykson-Zuber integral \cite{IZ}
and it 
 is valid only for Gaussian divisible ensembles with
   unitary  invariant Gaussian component.  
The size of the 
Gaussian component was reduced to $N^{-1+ \e}$ in \cite{EPRSY} by using 
an improved formula for correlation functions and 
the local semicircle law from  \cite{ESY1, ESY2, ESY3}. 

To eliminate the usage of an explicit formula, 
a conceptual approach for Step 2   via the local ergodicity of
 Dyson Brownian motion was    initiated  in \cite{ESY4}.  
In this paper, the first version of the local  relaxation flow was
introduced, but it was rather complicated. 
 In \cite{ESYY} we found a much simpler way to enhance the convexity of the Dyson
Brownian motion and we proved
a general theorem for local ergodicity of  DBM 
and related flow, i.e., 
Theorem  \ref{thm:DBM}. 
 This theorem applies to  all classical 
ensembles, i.e.,  real and complex Wigner matrices,   real and complex 
 sample covariance  matrices  and quaternion Wigner matrices. 
The local relaxation   flow in the simple form  \eqref{tl}  first appeared in \cite{ESYY}.  
 The relaxation time to local equilibrium proved in these two papers was not optimal;
 the optimal relaxation time, conjectured by Dyson,  was obtained
later in \cite{EYYrigi}.

The  third and final step is to approximate the local eigenvalue 
distribution of  a general Wigner matrix by that of a Gaussian divisible one. The first approximation 
result was obtained via the reversal heat flow in \cite{EPRSY}
 which required some smoothness of the distribution 
of matrix elements. Shortly after,  Tao and Vu  \cite{TV}, 
proved  a comparison theorem with a four moment matching condition.   
Instead of using a Gaussian divisible ensemble  with a small ($N^{-1 + \e}$) Gaussian 
component, they  relied  on    Johansson's result  \cite{J} to provide   Hermitian Gaussian 
divisible ensembles  for comparison.
This  proved the  universality of  Hermitian  Wigner matrices, provided that 
the  distributions of matrix elements have  vanishing third moment and 
  are supported on at least three points.
These 
conditions were removed in \cite{ERSTVY} by 
combining the arguments 
 of \cite{EPRSY} and \cite{TV}.

 Due  to the lack  of a  Br\'ezin-Hikami  type  formula for the  symmetric matrices, 
there was no extension of Johansson's result  \cite{J} to this case and the universality for
 symmetric Wigner ensembles  was 
much more difficult to prove.  However,  the 
result of \cite{TV}   implies that  the local eigenvalue statistics of symmetric 
Wigner matrices
and  GOE are the same,   but under the restriction that 
the first four moments of the matrix elements  exactly match those of GOE.
The resolution of the Wigner-Dyson-Gaudin-Mehta conjecture for symmetric matrices, i.e., 
Theorem~\ref{thm:DBM} for real symmetric matrices, 
was obtained in \cite{ESY4, EYY}. In these papers, two  new ideas were introduced: 
the local  relaxation flow  \cite{ESY4} 
 and   the Green function comparison theorem \cite{EYY}. 
 Starting from the paper \cite{EYY},  the variances 
were allowed to vary and the universality was extended to  generalized Wigner matrices. 
The real Bernoulli random 
matrices required a more refined argument \cite{EYY2}. 
Finally, the technical condition assumed in all these papers, i.e., 
that the probability distributions of the matrix 
elements decay subexponentially, 
was reduced to the $(4+ \e)$-moment assumption \eqref{4+e} by using  the
universality of Erd\H{o}s-R\'enyi matrices \cite{EKYY2}.

The Green function comparison theorem, Theorem \ref{comparison}, 
 uses the same four moment conditions which  appeared  earlier in  \cite{TV}, but it 
compares matrix elements of Green 
functions at a fixed energy  and not just traces of Green functions which carry information
 on  eigenvalues near a fixed energy.   The result of \cite{TV}, on the other hand,  concerns
individual eigenvalues with  fixed labels. Both proofs 
used the local semicircle law and  Lindeberg's idea (introduced  in 
his proof of the central limit theorem).
Lindeberg's idea  in the context of random matrices
 appeared  earlier in a proof of  the Wigner semicircle law by  Chatterjee \cite{Ch}. 
The approach \cite{TV} requires additional difficult estimates due to singularities
 from neighboring eigenvalues,  
but  the  Green function comparison theorem  follows directly from the local semicircle law 
in Step 1, i.e.,  Theorem \ref{lsc},  via standard resolvent expansions. 
The difficulties associated with the singularities
of eigenvalue resonances  are  completely absent  in the Green function comparison theorem.  
 Finally, we mention that Green function comparison can also yield comparison of eigenvalues 
with fixed labels, see the recent work by Knowles and Yin \cite{KY}.

The edge universality for Wigner matrices was first  proved 
via the moment method by Soshnikov \cite{Sosh} (see also the earlier work \cite{SS})
for Hermitian and symmetric   ensembles with  symmetric distributions.
By combining the moment 
method and Chebyshev polynomials, Sodin \cite{So1, So2}
proved edge universality of certain band matrices  and some special class
of sparse matrices 
  with symmetric  distribution. 
The symmetry assumption  was partially removed in 
\cite{P1, P2}.
 The edge universality
without any symmetry assumption
 was proved in \cite{TV2} under the condition 
that the distribution of matrix element is subexponential  decay and  the first 
three moments match those of  a Gaussian distribution.
  The subexponential decay condition is not optimal for
edge universality, in fact the finiteness of the fourth moment was conjectured to be sufficient.
For  Gaussian divisible  Hermitian ensembles this was proved in \cite{Joh}.
 This is optimal, since
on the other hand, the result by  Auffinger, Ben Arous and P\'ech\'e  \cite{ABP}  
 showed that the distribution of the largest eigenvalues converges to
 a Poisson process if the entries 
have at most $4-\e$ moments.
For Wigner matrices with arbitrary symmetry class, 
the edge universality was proved under the 
 sole assumption that the matrix entries 
have $12 + \e$ moments \cite{EKYY2}. Finally, we mention that extension of universality to eigenvectors near the edge 
was obtained by Knowles and Yin \cite{KY} under the two moment matching condition and with four  moment matching  
condition in \cite{TV4}.

Although we have focused only on Wigner matrices and
$\beta$-ensembles,
the ideas summarized in this review should be applicable to a wide class of  matrix ensembles. 
We have already mentioned some natural open questions related to possible improvements
of our results. These concern removing some technical  conditions
such as (i) the restriction $q\gg N^{1/3}$ in the bulk universality of the Erd{\H o}s-R\'enyi
matrix; (ii) the  $12+\e$ moment condition  for edge universality.
A more ambitious goal would be to prove universality for systems with some spatial
structure such as band matrices or  related models that may open up a 
path towards universality for random Schr\"odinger operators and other
realistic models of quantum chaos.

\end{document}